\documentclass[11pt,a4paper]{amsart}
\usepackage{amsmath,amsthm,amssymb}
\usepackage{latexsym}

\allowdisplaybreaks
\numberwithin{equation}{section}

\newtheorem{theorem}{Theorem}[section]
\newtheorem*{theoremm}{Theorem}
\newtheorem{lemma}{Lemma}[section]
\newtheorem{corollary}{Corollary}[section]
\newtheorem*{corollaryy}{Corollary}
\newtheorem{proposition}{Proposition}[section]
\newtheorem{definition}{Definition}[section]

\def\cal{\mathcal}
\let\Re=\undefined
\DeclareMathOperator{\Re}{Re}
\let\Im=\undefined
\DeclareMathOperator{\Im}{Im}

\newcommand{\nn}{\nonumber}
\newcommand{\nt}{\noindent}

\newcommand{\bsl}{\backslash}

\newcommand{\pt}{\partial}
\newcommand{\ti}{\tilde}

\newcommand{\lt}{\left}
\newcommand{\rt}{\right}

\newcommand{\dsp}{\displaystyle}

\DeclareMathOperator{\diam}{\rm diam}

\DeclareMathOperator{\ppr}{Pr}

\newcommand{\br}{{\mathbb{R}}}

\newcommand{\bz}{{\mathbb{Z}}}
\newcommand{\bc}{{\mathbb{C}}}

\newcommand{\bi}{{\mathbf{i}}}
\newcommand{\bii}{{\mathbf{I}}}

\renewcommand{\nn}{{\nabla}}

\renewcommand{\l}{\lambda}

\newcommand{\vp}{\varphi}

\newcommand{\la}{\langle}
\newcommand{\ra}{\rangle}

\newcommand{\cp}{\mathcal{P}}
\newcommand{\ca}{\mathcal{A}}
\newcommand{\ccc}{\mathcal{C}}
\newcommand{\cw}{\mathcal{W}}

\begin{document}

\title[It\^o diffusions, capacity and Schr\"odinger operator]
{It\^o diffusions, modified capacity and harmonic measure.
Applications to Schr\"odinger operators.}

\author[S. Denisov, S. Kupin]{S. Denisov, S. Kupin}

\address{Mathematics Department, University of Wisconsin--Madison,   480 Lincoln Dr.,  Madison, WI 53706 USA}
\email{denissov@math.wisc.edu}

\address{IMB, Universit\'e Bordeaux 1, 351 cours de la Lib\'eration, 33405 Talence Cedex France}
\email{skupin@math.u-bordeaux1.fr}




\begin{abstract}
We observe that some special It\^o diffusions are related to
scattering properties of a Schr\"odinger operator on $\br^d, d\geq
2$. We introduce Feynman-Kac type formulae for these stochastic
processes which lead us to results on the preservation of the a.c.
spectrum of the Schr\"odinger operator. To better understand the
analytic properties of the processes,  we construct and study a
special version of the potential theory.  The modified capacity and
harmonic measure play an important role in these considerations. We
also give various applications to Schr\"odinger operators.
\end{abstract}

\maketitle

\section*{Introduction}\label{s0}
The main motivation for this paper was  an exciting open problem in the spectral
theory of a multi-dimensional Schr\"o\-din\-ger operator. Namely, we are interested in the spectral properties of the operator
\begin{equation}\label{SO}
H=H_V = -\Delta +V
\end{equation}
 acting on $L^2(\br^d), d\geq 2$.  One of the central problems of the topic is to
 understand under what conditions on $V$ the absolutely continuous (a.c., to be brief) spectrum $\sigma_{ac}(H)$ is present and what its essential support is. Putting this a bit differently, we want to know what perturbations $V$ preserve the a.c. spectrum as compared to $H_0=-\Delta$. Of course,
$\sigma(H_0)=\sigma_{ac}(H_0)=\br_+$. The related physical intuition is that if the potential decays sufficiently fast at infinity, there is some a.c. spectrum which means that the scattering properties of the medium modeled by the operator are not too bad.

For instance, Deift-Killip \cite{dk1} proved in one-dimensional case
that the condition $V\in L^p(\br_+)$ leads to $\sigma_{ac}(H)=\br_+$
if $p\leq 2$. If $p>2$, there are some $V$ such that the spectrum is
singular \cite{kls}. For  $d\geq 2$, a counterpart of the result by
Deift and Killip
 was conjectured by Simon \cite{rarry}. His
question is: does
\begin{equation}\label{e10001}
\int_{\mathbb{R}^d} \frac{V^2(x)}{|x|^{d-1}+1}\,dx<\infty
\end{equation}
imply that $\sigma_{ac}(H)=\br_+$?  Here one might also need more regularity of the potential $V$ for the operator $H_V$ to be well-defined,  but this is a minor issue.

While being wide open for the case $\br^d$ or $\bz^d, d\geq 2$, the problem has a neat solution on a Cayley tree (Bethe lattice), see Denisov \cite{d2}, Denisov-Kiselev \cite{d1}. Since this construction serves as a starting point to the present investigation, it seems instructive to recall some of its details.

Assume that the Cayley tree $\mathbb{B}$ is rooted with the root (the
origin) denoted by $O$, $O$ has two neighbors and other vertices
have one ascendant and two descendants (the actual number of descendants is not important but it should be the
same for all points $X\neq O$). The set of vertices of the tree is denoted by $\mathbb{V}(\mathbb{B})$. For an $f\in \ell^2(\mathbb{V}(\mathbb{B}))$, define the free Laplacian by
$$
({H}_0f)_n=\sum_{{\rm dist}(i,n)=1}f_i, \quad n\in \mathbb{V}(\mathbb{B})
$$
One can show rather easily \cite[Sect. 2]{d2} that the spectrum of ${H}_0$ is purely a.c. on $[-2\sqrt 2, 2\sqrt 2]$.
Assume now that $V$ is a bounded potential on $\mathbb{V}(\mathbb{B})$ so that
$$
{H}={H}_0+V
$$
is well-defined. Denote the spectral measure related to delta
function at $O$ by $\sigma_O$; the density of its absolutely continuous part  is  $\sigma'_O$. Take $w(\lambda)=(4\pi)^{-1}(8-\lambda^2)^{1/2}$ and let
$\rho_O(\lambda)=\sigma'_O(\lambda)w^{-1}(\lambda)$.

Consider also the probability space on the set of nonintersecting paths in
$\mathbb{B}$ that go from the origin to infinity. This space is
constructed by assigning the Bernoulli random variable to each
vertex and the outcome of Bernoulli trial ($0$ or $1$) then
corresponds to whether the path (stemming from the origin) goes
to the ``left" or to the ``right" descendant at the next step. Notice
also that (discarding a set of Lebesgue measure zero) each path is
in one-to-one correspondence with a point on the interval $[0,1]$ by
the binary decomposition of reals. In this way, the ``infinity"
for $\mathbb{B}$ can be identified with $[0,1]$. For any $t\in
[0,1]$, we can then define the function $\phi$ as
$$
\phi(t)=\sum_{n=1}^\infty V^2(x_n)
$$
where the path $\{x_n\}\subset \mathbb{V}(\mathbb{B})$ corresponds
to $t$. This function does not have to be finite at any point $t$
but it is well-defined and is Lebesgue measurable.  See \cite{d1}
for
\begin{theorem}\label{tree}
For any bounded $V$,
\begin{eqnarray*}\label{bs}
\int\limits_{-2\sqrt 2}^{2\sqrt 2} w(\lambda)\log
\rho_O(\lambda)d\lambda &\ge& \log \mathbb{E} \left\{ \exp\left[
-\frac 14 \sum\limits_{n=1}^\infty V^2(x_n)\right]\right\}\\
&=&\log \int\limits_0^1\exp\left(-\frac 14\phi(t)\right)dt
\end{eqnarray*}
where the expectation is taken with respect to all paths
$\left\{x_n\right\}$ and the probability space defined above. In
particular, if the right hand side is finite, then $[-2\sqrt
2,2\sqrt2]\subseteq \sigma_{ac}({H})$.
\end{theorem}
The proof of  the theorem is based on the adjusted form of sum rules in the spirit of Killip-Simon \cite{ks}. Higher order sum rules are applied to different classes of potentials in Kupin \cite{ku}.

Notice that $\phi$ is always nonnegative, therefore the right hand
side is bounded away from $-\infty$ iff $V\in \ell^2$ with a
positive probability. This is the true multi-dimensional
$L^2$-condition. The simple application of Jensen's inequality then
immediately implies that the estimate
\[
\int \phi(t)dt=\sum\limits_{n=0}^\infty 2^{-n} \sum\limits_{{\rm
dist}(X,O)=n} V^2(X)<\infty
\]
guarantees  $[-2\sqrt 2,2 \sqrt2]\subseteq \sigma_{ac}({H})$. The
last condition is precisely the analogue of (\ref{e10001}) for the
Cayley tree. Indeed, the factor $2^{n}$ is the ``area" of the sphere
of radius $n$ in $\mathbb{B}$ and is exactly the counterpart of
$|x|^{d-1}$ in (\ref{e10001}). \medskip

\nt{\bf Remark.} Assume that in the above model there is a set $F$ of vertices on which $V$ is uncontrollable and $V=0$ on $\mathbb{V}(\mathbb{B})\backslash  F$. Then, the theorem says that the a.c. spectrum of $H$ contains $[-2\sqrt 2, 2\sqrt 2]$ as long as there are ``enough'' paths that do not visit $F$.

\medskip
The substantial part of this paper is devoted to the study of analogous phenomena in the case of $\mathbb{R}^d$ (see, e.g., Theorem \ref{eas}).  In its first part,  we obtain a result similar to
Theorem~\ref{tree} for a Schr\"odinger operator \eqref{SO}. Besides
various technical difficulties, an immediate problem we run into  is
the question of how to introduce a probability space of paths in
$\br^d$ similar to the one appearing in \eqref{bs}.

It turns out rather naturally that the right probability space of paths is given by
 It\^o's stochastic calculus or, more precisely, by a (stationary) It\^o stochastic differential equation of the form
$$
dX_t=p(X_t)dt+ dB_t,\quad X_0=x^0
$$
where $B_t$ is a $\br^d$--Brownian motion. The solution $\{X_t\}$ to
this equation is called  the It\^o diffusion. The coefficient $p$ is
termed a drift. We refer to the nice books by \O ksendal \cite{Oks}
and Bass \cite{baas} in this connection; see also Section~\ref{s1}
for more details.

Consider now the diffusion $\{X_t\}$ defined by \eqref{stochastic}.
As a corollary of Feynman-Kac type formulae proved in
Section~\ref{s2}, we get the following theorem.

\begin{theoremm}[{= Theorem \ref{th2}}]\label{th222}
Let $V\geq 0$  be continuous, $f\in L^2(\mathbb{R}^d), f\geq 0$ and
$f$ have a compact support. Let $\sigma_f$ be the spectral measure
of $f$ with respect to $H_V$. We have
\begin{equation}
\exp\left[\frac{1}{2\pi}\int\limits_{-\infty}^\infty \frac{\log
\sigma'_f(k^2)}{1+k^2}dk\right]\ge C_f\int f(x^0)\mathbb{E}_{x^0}
\left[ \exp\left(-\frac 12\int\limits_0^\infty V(X_\tau)
 d\tau\right)\right]dx^0 \label{e1002}
\end{equation}
and  the constant $C_f>0$ does not depend on $V$.
\end{theoremm}
Above, $\sigma'_f$ stands for the density of the absolutely
continuous component of $\sigma_f$. This theorem can be viewed as a
counterpart of Theorem \ref{tree} although the $L^2$-summability
over the path is replaced by the stronger $L^1$-condition. Here is
one of its corollaries.
\begin{corollaryy}[{=Corollary \ref{c01}}]\label{c011}
Let $V$ and $f$ be as above. Then
\[
\int\limits_\br \frac{\log
\sigma'_f(k^2)}{1+k^2}dk\ge-C_{1,f}\int_{\br^d} \frac{V(x)}{|x|^{d-1}+1}dx-C_{2,f}
\]
\end{corollaryy}

It is appropriate to compare this to results of Laptev-Naboko-Safronov \cite{lns1, lns2}.
 For a potential $V$, we write $V_\pm=\max(\pm V, 0)$ so that $V=V_+-V_-, |V|=V_++V_-$.
 Let $\Delta_D$ be the Laplace operator on $\mathbb{R}^d\backslash B(0,1)$ with the
 Dirichlet boundary condition on the unit sphere.
 Take $H=-\Delta_D+V$.
 We have
\begin{theorem}[{\cite{lns1}}]\label{lns} Let $V\in \ell^\infty(\bz^d, L^q(\Pi_1)), q>d/2,$ and
$$
\int_{\br^d} V_-^{(d+1)/2}(x) dx<\infty, \quad \int_{\br^d} \frac{V_+(x)}{|x|^{d-1}+1} dx<\infty
$$
Then
$$
\int_{\br_+} \frac{\log \sigma'_f(E)}{(1+E^{3/2})E^{1/2}} dE>-\infty
$$
The function $f \in L^2(\br^d)$ is bounded, spherically symmetric
and has compact support, $\Pi_1=[0,1]^d\subset \br^d$.
\end{theorem}

The discovery of deep and fruitful relations between the Brownian motion and
 Schr\"odinger operator probably goes back as far as to Wiener \cite{wi}, Feynman
 \cite{fey} and  Kac \cite{ka1, ka2}. Chung-Zhao \cite{chzh} use  these ideas to
 make connections with gauge theory, gaugeability, the properties of the discrete
 spectrum and eigenfunctions of the Schr\"odinger operator. Aizenman-Simon
 \cite{ais} apply this technique to study the properties of the
 evolution semigroup of $H_V$; see Simon \cite{si2} for an overview of the topic.
 The new aspects of our approach are that, first, we work with an ``appropriately
 modified Brownian motion" (= the It\^o diffusion), and not its ``classical" version.
 Second, this allows us to bypass several steps of computations and to get rather
 directly to the spectral measure of $H_V$ and, especially, its a.c. component.

Further analysis requires a good understanding of stochastic
integrals appearing in the right hand side of \eqref{e1002}.  To simplify the picture, we then assume that $V=0$ on a
domain $\Omega\subset \br^d$ and $V$ is of arbitrary size on
$E=\Omega^c$. The second part of the article deals with stochastic
integrals \eqref{e1002} from the point of view of the potential
theory induced by $\{X_t\}$. To give an idea of the results obtained
in this direction, we give some definitions for $d=3$;
Section~\ref{ss31} contains more details. Consider operators
$$
L^\pm=\frac12 \Delta\mp\pt_{x_1}
$$
We say that a function $u$ is $L^\pm$-harmonic on $\Omega$, if $L^\pm u=0$. Since $L^\pm$ are second order elliptic operators, one can show that $L^\pm$-harmonic functions possess many usual properties of harmonic functions (i.e., the max/min principle, Harnack principle etc., see Landis \cite{landis}).

We want to build the potential theory for these operators. Modulo some technical aspects, the construction follows the lines of the classical case related to $L=\Delta$; see, for instance, Landkof \cite{lan}, Hayman-Kennedy \cite{hk} and Garnett-Marshall \cite{GM}. So, let $z=(z_1,z'), \xi=(\xi_1,\xi')\in\br^3$ with $z,\xi'\in\br^2$.
We introduce the potentials
\begin{equation*}
K^\pm(z,\xi)=2G_0(z,\xi) e^{\pm(z_1-\xi_1)}
\end{equation*}
where $G_0$ is Green's function for $(-\Delta+1)^{-1}$ on
$\mathbb{R}^3$, that is
\[
G_0(z,\xi)=G_0(z,\xi; i)=\frac 1{4\pi}\frac{e^{-|z-\xi|}}{|z-\xi|}
\]
Let $E$ be a compact subset of $\br^3$ and $\cp(E)$ be the set of probability Borel measures on $E$. For $\mu\in\cp(E)$, set
\[
U_{\mu}^\pm(z)=\int_E K^\pm(z,\xi)d\mu(\xi)
\]
to be the corresponding potential. Consider
\begin{equation*}
C^\pm(E)=\lt(\inf_{\mu\in \cp(E)} \sup_{z\in \mathbb{R}^3}
U_{\mu}^\pm(z)\rt)^{-1} \label{def1_}
\end{equation*}
It turns out that, by Theorem \ref{th_trii}, $C^\pm(E)$ are equal
and their common value is denoted by $C(E)$. We call $C(E)$ the
modified capacity of $E$. More involved results along with some
developments of this theory are in Section~\ref{s3}.

In Section \ref{s5}, we give applications of the introduced
techniques to the spectral theory of Schr\"odinger operator. For
example, we construct an obstacle $E$ with the  following
properties:
\begin{itemize}
\item $\mathbb{P}(X_t$ does not hit $E)>0$ and hence $\sigma_{ac}(H)=\br_+$.
\item any ray, issued from the origin, intersects $E$ infinitely many times.
\end{itemize}
It is interesting to compare this example to the result by
Amrein-Pearson \cite{pears}, where the authors show that
$\sigma_{ac}(H)=\br_+$ if there is a sufficiently thick
obstacle--free cone. Our results suggest that the phenomenon of the
preservation of the a.c. spectrum is much finer and is of  a
capacitary nature. In particular, we present an example where the
above ``cone condition" is not satisfied but the a.c. spectrum is
nevertheless preserved.  The Appendices~A and B contain some bounds
on modified harmonic measure with respect to the operators $L^\pm$
introduced above.

We conclude the introduction with a few words on the notations.
We write $B(x,r)=\{y :|y-x|<r\}, r>0$, for an open ball in $\br^d$, and $int(A)$ is
the interior of an $A\subset \br^d$.
$\Sigma_r$ is a sphere of radius $r>0$ in $\br^d$. For a domain
$\Omega$ in $\br^d$, $L^p(\Omega)$ is the usual Lebesgue space and
$$
W^{1,2}(\Omega)=\{f: \int_\Omega (|f|^2+|\nabla f|^2) dx<\infty\}
$$
the derivatives being understood in the distributional sense. We will write that
\[
f(x)\approx g(x), \quad x\to x_0
\]
if
\[
\lim\limits_{x\to x_0} \frac{f(x)}{g(x)}=1
\]
and
\[
f(x)\sim g(x), \quad {\rm for} \quad x\in I
\]
if
\[
C_1<\frac{f(x)}{g(x)}<C_2, \quad x\in I
\]
with some $C_1, C_2>0$. We also write $f(x)\lesssim g(x)$ on $I$ if $f(x)<Cg(x)$ for $x\in I$ with some $C>0$.
The
probabilistic notations are largely borrowed from \O ksendal
\cite{Oks}, see also Karatzas-Shreve \cite{kara}.

\section{A stochastic differential equation}\label{s1}

The following
stochastic differential equation will play an important role later
on. Consider the Lipschitz vector field
\[
 p(x)=\left(\frac{I'_\nu(|x|)}{I_\nu(|x|)}-\nu |x|^{-1}\right)\cdot
\frac{x}{|x|}, \quad \nu=(d-2)/2
\]
where $I_\nu$ denotes the modified Bessel function \cite[Sect.
9.6]{as}. The asymptotics of $I_\nu, I'_\nu$ at zero and at infinity
are given by Abramowitz-Stegun \cite{as}, formulae (9.6.10),(9.7.1),
and (9.7.3). They yield
\begin{eqnarray}
p(x)&=&\left(1-\left(\nu+\frac 12\right)
|x|^{-1}+\underline{O}(|x|^{-2})\right)\cdot x|x|^{-1}, \quad |x|\to
\infty \label{e1}\\
p(x)&=&\left(\frac{|x|}{2(\nu+1)}+\underline{O}(|x|^3)\right)\cdot
x|x|^{-1}, \quad |x|\to 0 \label{e2}
\end{eqnarray}
Then, fix any point $x^0\in\mathbb{R}^d$ and consider the following
stochastic process
\begin{equation}
dX_t=p(X_t)dt+dB_t, \quad X_0=x^0 \label{stochastic}
\end{equation}
with the drift given by $p$. The solution to this diffusion
process exists and all trajectories are continuous almost surely.
Its generating operator \cite[Sect. 7.3]{Oks} is given by
\begin{equation}\label{e4}
A=\frac 12\Delta + p\cdot\nabla
\end{equation}

We need to understand better the properties of the trajectories $X_t$.
Consider the radial component $Z_t=|X_t|$.  By
Ito's formula \cite[Sect. 4.2]{Oks}, we have
\[
dZ_t=d\tilde{B}_t+\left(
\frac{1}{2Z_t}+\frac{I_\nu'(Z_t)}{I_\nu(Z_t)} \right)dt
\]
where $\tilde B_t$ is a one-dimensional Brownian motion
\cite[Sect. 8.4]{Oks}. The generating operator for $Z_t$ \cite[Sect. 7.3]{Oks} is given by
\[
D=\frac 12\frac{d^2}{dx^2}+\alpha(x)\frac{d}{dx}
\]
where
\[
\alpha(x)=\frac{1}{2x}+\frac{I_\nu'(x)}{I_\nu(x)}
\]
From  \eqref{e1}, we get $\alpha(x)=1+\underline{O}(x^{-2})$ for $x\to\infty$
so all paths of $Z_t$ go to infinity almost surely.
Notice that
$$
p(x)=\nabla \log \left(|x|^{-\nu}I_\nu(|x|)\right)
$$
and consider $Q=|p|^2+\mathrm{div}\, p$. Recall that (\cite{as},
formula (9.6.1)),
\[
r^2I_\nu''(r)+rI_\nu'(r)=(r^2+\nu^2)I_\nu(r)
\]
Since
\[
\Delta f(|x|)=\frac{1}{r^{d-1}}\partial_r\left(r^{d-1}\partial_r
f\right)
\]
with $r=|x|$, we have
$$
\Delta(|x|^{-\nu}I_{\nu}(|x|))=|x|^{-\nu}I_{\nu}(|x|)
$$
An easy computation shows that $Q=1$ on $\mathbb{R}^d$. Consider now a self-adjoint semigroup given by
\begin{equation}\label{e3}
\psi_t=\frac 12 \Delta\psi-\frac Q2\psi
\end{equation}
Its transition probability is
\[
\hat{p}(x,y,t)=\frac{1}{(2\pi t)^{d/2}}e^{-\frac{|x-y|^2}{2t}-\frac
t2}
\]
On the other hand, let $F(x)=|x|^{-\nu} I_\nu(|x|)$ and $\psi=F\,\phi$. Notice that $\psi$ satisfies \eqref{e3} iff
\begin{equation}\label{e8}
\phi_t=\frac 12\Delta \phi+ p\cdot \nabla\phi
\end{equation}
The operator appearing in the right hand side of the above equality
is precisely $A$ from \eqref{e4}. So, due to the connection between
the diffusion and the processes with killing \cite[Ch. 8]{Oks},
Exercise 8.16, we have
\[
p(x,y,t)=\frac{|y|^{-\nu}I_\nu(|y|)}{|x|^{-\nu}I_\nu(|x|)}\, \hat{p}(x,y,t)
\]
where $p(x,y,t)$ is the transition probability for $X_t$. Once
again, the asymptotics (9.7.1) from \cite{as} implies
\begin{equation}
p(x,y,t)\sim \frac{1}{(2\pi t)^{d/2}} \left( \frac{|x|}{|y|}
\right)^{(d-1)/2}\exp\left(-\frac{|x-y|^2}{2t}- \frac t2
+|y|-|x|\right) \label{tr-pr}
\end{equation}
as $|x|,|y|$ are large.

\section{From It\^ o calculus to spectral properties of a Schr\"odinger operator}\label{s2}
\subsection{Feynman-Kac type formulae}\label{ss21} We start this subsection by introducing some notations. Let $H_V$ be
a Schr\"odinger operator \eqref{SO} and  assume the potential $V$ is
nonnegative, continuous, and has a compact support.   The Green's
function  $G=G_V$ of the operator $H_V$ is defined by
\begin{equation}\label{e82}
 ((H-k^2)^{-1}f)(x)=\int_{\br^d} G_V(x,y;k)f(y)\, dy
\end{equation}
where $f\in L^2(\br^d)$, $k\in\bc_+$, and $x,y\in\br^d$. Recall that for $V=0$,
\[
G_0(x,y;k)=C_d' (-ik)^\nu\frac{ K_\nu (-ik|x-y|)}{|x-y|^\nu}
\]
where $\nu=(d-2)/2$, $C_d'=1/(2\pi)$ for $d=2$,  and
$C_d'=1/(4(2\pi)^{d-2})$ for $d\geq 3$. The asymptotics yields
$$
G_0(x,y;k)\approx \tilde C_d\left\{
\begin{array}{ll}
-\log |x-y|, & d=2,\\
|x-y|^{-(d-2)}, & d\geq 3,
\end{array}
\right. \quad x\to y
$$
where $\ti C_d=1/(2\pi)$ for $d=2$, $\ti C_d=\Gamma(d/2)/(2\pi^{d/2} (d-2))$ for $d\geq 3$. We also have
\begin{equation}
G_0(x,y;k)\approx C_d (-ik)^{\nu-1/2}\frac{e^{ik|x-y|}}{|x-y|^{(d-1)/2}},
\quad |x|\to\infty\label{vtor}
\end{equation}

Since $V$  is compactly supported, we have the following relations (e.g., Denisov \cite{d3}):
\begin{itemize}
\item For a fixed $x^0$, define the amplitude $a_{x^0}$ as
\begin{equation}\label{as}
G(x,x^0;k) \approx  C_d
(-ik)^{\nu-1/2}\frac{e^{ik|x|}}{|x|^{(d-1)/2}}\, a_{x^0}(\theta, k),
\end{equation}
where $\theta=x/|x|\in\Sigma_1$ and $|x|\to \infty$, the constant
$C_d$ is from (\ref{vtor}).
\item
Furthermore,
\begin{equation}\label{e84}
G(x,x^0;k) \approx C_d
(-ik)^{\nu-1/2}\frac{e^{ik|x-x^0|}}{|x-x^0|^{(d-1)/2}}\,
\beta_{x^0}(\theta, k),
\end{equation}
where $k\in \mathbb{C}^+$ and $|x|\to\infty$.
\item
We have $|x|=\la x^0, \theta\ra +|x-x^0|+\bar{o}(1)$ as $|x|\to\infty$ and
 consequently
\begin{equation}\label{e85}
 a_{x^0}(\theta, i)=\beta_{x^0}(\theta,i)e^{-ik\la x^0,\theta\ra}
\end{equation}
\item Consider the function $u(.,k)=(H-k^2)^{-1}f$, where $f\in L^2(\mathbb{R}^d)$ and has a compact support.
The function $u$ has the following asymptotics
\[
u(x,k)\approx C_d
(-ik)^{\nu-1/2}\frac{e^{ik|x|}}{|x|^{(d-1)/2}}A_{f}(\theta,k)
\]
as $|x|\to\infty$, $\theta=x/|x|$. For $A_f$, we have
\begin{eqnarray}\label{ba}
A_f(\theta,k)&=&\int e^{-ik\langle x^0,\theta\rangle}\beta_{x^0}(\theta,k)f(x^0)\, dx^0 \\
&=&\int a_{x^0}(\theta,k) f(x^0)\, dx^0 \nonumber
\end{eqnarray}
\end{itemize}

The next result gives probabilistic interpretation for the average of the amplitude
$a_{x^0}=a_{x^0}(. ,i)$ over the unit sphere $\Sigma_1$. We denote by $\{e^k\}_k$ be the standard basis in $\br^d$ and, for $x\in\br^d$, we write
$x=\sum_k x_ke^k=x_1e^1+x'$; so $x=(x_1,x')$.

\begin{theorem}(Feynman--Kac type formula)\label{th1}
\begin{enumerate}
\item
Let $X_t$ be the solution to (\ref{stochastic}). Then we have
\begin{equation}
\int\limits_{\Sigma_1} a_{x^0}(\theta)d\theta=C_1\mathbb{E}_{x^0}
\left[ \exp\left(-\frac12 \int\limits_0^\infty
V(X_\tau)d\tau\right)\right] \label{feynman-kac}
\end{equation}
where the subscript $x^0$ means that the process starts at
$x^0$.
\item
For $\theta\in\Sigma_1$, let $dG_t=\theta dt+dB_t$. Then
\begin{equation}
a_{x^0}(\theta)=C_2\mathbb{E}_{x^0} \left[ \exp\left(-\frac12
\int\limits_0^\infty V(G_\tau)d\tau\right)\right]\label{fk}
\end{equation}
\end{enumerate}
\end{theorem}

\begin{proof}
Let $V$ have support inside the ball $\{|x|<\rho\}$. Take $r$ and
$R$ such that $\rho\ll r\ll R$ and consider the solution to the
following problem
\begin{equation}
-\Delta \psi+V\psi=-\psi+f,\quad \psi|_{\Sigma_R}=0\label{eqeq}
\end{equation}
where $f$ is characteristic function of the spherical layer $\{r<|x|<r+1\}$.
Then, we have two expressions for $\psi(x^0)$. On the one hand,
\[
\psi(x^0)=\int\limits_{|x|<R}G^{(R)}(x^0,y)f(y)dy
\]
where $G^{(R)}(x,y)$ is the Green's function for the Schr\"odinger
operator in $\{|x|<R\}$ with Dirichlet boundary condition. On the
other hand, we substitute
\begin{equation}
\psi=F\, \phi  \label{eqz}
\end{equation}
where, just as before, $F(x)=|x|^{-\nu}I_{\nu}(|x|)$. The
asymptotics of $I_\nu$ near zero is discussed in the beginning of
Section \ref{s1}. It implies that $F$ is infinitely smooth on
$\br^d$ and $F \ge C>0$.  So, recalling \eqref{e8}
\[
\frac 12 \Delta\phi+ p\nabla \phi-\frac 12 V\phi =-\frac 12 f F^{-1}
\]
and $\phi=0$ on $ \Sigma_R$.  We have the following representation
\cite[Ch. 9]{Oks}, Exercise~9.12,
\[
\phi(x^0)=\mathbb{E}_{x^0}\left[\int\limits_0^{T_R} \exp\left(-\frac
12\int\limits_0^\tau V(X_s)ds\right)f_1(X_\tau)d\tau \right]
\]
where $f_1(x)=\frac 12 f F^{-1}$ and $T_R$ is the standard stopping time,
i.e., the random time when $X_t$ hits the boundary $\Sigma_R$ for the
first time. The path $X_t$ is the solution to stochastic
differential equation (\ref{stochastic}). Thus we have the identity
\begin{equation}
\frac{e^{ r}}{r^{(d-1)/2}}\int\limits_{r<|x|<r+1}
G^{(R)}(x^0,y)f(y)dy=\mathbb{E}_{x^0}\left[\int\limits_0^{T_R}
\exp\left(-\frac 12\int\limits_0^\tau
V(X_s)ds\right)f_2(X_\tau)d\tau \right] \label{id}
\end{equation}
where
\[
f_2(x)=f_1(x)\frac{e^{r}}{r^{(d-1)/2}}
\]
It is well known that $G^{(R)}(x,y)$ tends to $G(x,y)$ uniformly for
$x,y$ from a fixed compact as $R\to\infty$. Therefore, asymptotics
(\ref{as}) yields
\[
\lim\limits_{r\to\infty}\lim\limits_{R\to\infty} \Bigl({\rm LHS\ of\
(\ref{id})}\Bigr)=C\int\limits_{\Sigma_1}a_{x^0}(\theta)d\theta
\]
Since $T_R$ goes to infinity as $R\to\infty$, we have
\begin{eqnarray*}
\mathbb{E}_{x^0}\left[\int\limits_0^{T_R} \exp\left(-\frac
12\int\limits_0^\tau V(X_s)ds\right)f_2(X_\tau)d\tau
\right] \longrightarrow \hspace{0.3cm}\\\hspace{1cm}
\mathbb{E}_{x^0}\left[\int\limits_0^{\infty} \exp\left(-\frac
12\int\limits_0^\tau V(X_s)ds\right)f_2(X_\tau)d\tau \right]
\end{eqnarray*}
Now, let us compute the limit of the last expression as
$r\to\infty$. It makes sense to consider the stopping time
$T_{r/2}$: the time when the path $X_t$ hits the sphere
$\Sigma_{r/2}$ for the first time. Hence,
\[
\mathbb{E}_{x^0}\left[\int\limits_0^{\infty} \exp\left(-\frac
12\int\limits_0^\tau V(X_s)ds\right)f_2(X_\tau)d\tau \right]=
\]
\[
\mathbb{E}_{x^0}\left[ \exp\left(-\frac 12\int\limits_0^{T_{r/2}}
V(X_s)ds\right)\int\limits_{T_{r/2}}^\infty \exp\left(-\frac
12\int\limits_{T_{r/2}}^\tau V(X_s)ds\right)f_2(X_\tau)d\tau \right]
\]
By the strong Markov property \cite[Sect. 7.2]{Oks}, the last expectation equals to
\[
\mathbb{E}_{x^0}\left[ \exp\left(-\frac 12\int\limits_0^{T_{r/2}}
V(X_s)ds\right) \mathbb{E}_{X_{T_{r/2}}}\left[
\int\limits_{0}^{\infty} \exp\left(-\frac 12\int\limits_{0}^\tau
V(\tilde{X}_s)ds\right)f_2(\tilde{X}_\tau)d\tau \right]\right]
\]
The process $\tilde{X}_t$ is the solution to (\ref{stochastic}) where $x^0=X_{T_{r/2}}$.

As $r\to\infty$, the inner expectation tends to a constant
independent of $X_{T_{r/2}}$. Since $T_{r/2}$ goes to infinity almost
surely, the dominated convergence theorem yields
(\ref{feynman-kac}).

The proof of the second claim of the theorem is similar. Without loss of generality, we assume $\theta=e^1$.
Let $f$ in (\ref{eqeq}) be given
by
\[
f=\chi_{\{r<|x|<r+1\}}\cdot \chi_{\{|x'|<\sqrt r\}}
\]
We now take $F(x)=e^{x_1}$ in (\ref{eqz}). Then,  we have
\[
\frac 12\Delta\phi+\phi_{x_1}-\frac 12 V\phi=-\frac 12 fe^{-x_1}
\]
and
\[
\phi(x^0)=\mathbb{E}_{x^0}\left[\int\limits_0^{T_R} \exp\left(-\frac
12\int\limits_0^\tau V(G_s)ds\right)f_1(G_\tau)d\tau \right]
\]
where $f_1(x)=\frac 12 f(x)e^{-x_1}$. We again have
\begin{equation}\label{e81}
\int\limits_{r<|x|<r+1}
e^{r}G^{(R)}(x^0,y)f(y)dy=\mathbb{E}_{x^0}\left[\int\limits_0^{T_R}
\exp\left(-\frac 12\int\limits_0^\tau
V(G_s)ds\right)f_2(G_\tau)d\tau \right]
\end{equation}
and $f_2(x)=\frac 12 f(x)e^{r-x_1}$. Take $R\to\infty$ first and
then $r\to\infty$. Definition of $a_{x^0}$ \eqref{as} implies
that the left hand side of the above equality will converge to a
multiple of $ a_{x^0}(e^1)$. For the right hand side, we have
\[
\lim_{R\to\infty} ({\rm RHS \ of \  \eqref{e81}
})=\mathbb{E}_{x^0}\left[\int\limits_0^{\infty} \exp\left(-\frac
12\int\limits_0^\tau V(G_s)ds\right)f_2(G_\tau)d\tau \right]
\]
Let $\epsilon$ be small positive ($\epsilon<1/2$ is enough) and $t_{r^\epsilon}$ be the first time when $G_\tau$
hits the plane $x_1=r^\epsilon$. Then,
\[
\mathbb{E}_{x^0}\left[\int\limits_0^{\infty} \exp\left(-\frac
12\int\limits_0^\tau V(G_s)ds\right)f_2(G_\tau)d\tau \right]=
\]
\[
\mathbb{E}_{x^0}\left[ \exp\left(-\frac
12\int\limits_0^{t_{r^\epsilon}} V(G_s)ds\right)
\mathbb{E}_{G_{t_{r^\epsilon}}}\left[ \int\limits_{0}^{\infty}
\exp\left(-\frac 12\int\limits_{0}^\tau
V(\tilde{G}_s)ds\right)f_2(\tilde{G}_\tau)d\tau \right]\right]
\]
The process $\tilde{G}_t$ again solves the same equation $d\tilde{G}_t=e^1dt+dB_t$ but with initial value $\tilde{G}_0=G_{t_{r^\epsilon}}$.

Let $|G'_{t_{r^\epsilon}}|<r^\epsilon$. Then, the inner expectation
converges to some $G_{t_{r^\epsilon}}$--independent constant as
$r\to\infty$.  On the other hand,
$\mathbb{P}(|G'_{t_{r^\epsilon}}|>r^\epsilon)\to 0$,  so the
remaining factor from the right hand side of the above equality
tends to
\[
\mathbb{E}_{x^0} \left[ \exp\left(-\frac12 \int\limits_0^\infty
V(G_\tau)d\tau\right)\right]
\]
\end{proof}

\nt{\bf Remark.} Given any continuous nonnegative potential $V$,
define truncations $V^{(n)}(x)=V(x)\cdot \mu_n(x)$ where $\mu_n(x)$
is smooth, equals to one on $|x|<n$ and to $0$ on $|x|>n+1$, and
$0\leq \mu_n(x)\leq 1$ everywhere. For each $n$, the above theorem
applies. The monotonicity of Green's function in $V$ and the
monotone convergence theorem allow one to show that formula
(\ref{feynman-kac}) is true for all continuous nonnegative
potentials. The amplitude $a_{x^0}(\theta)$ is well-defined as the
$\lim_{n\to\infty}a^{(n)}_{x^0}(\theta)$ where $a^{(n)}_{x^0}$
corresponds to $V^{(n)}$. Of course, in this case both
expressions can be equal to $0$.

\subsection{Applications to the scattering theory of  Schr\"odinger operators}\label{ss22}
We apply the methods of the previous subsection to the study of the a.c. spectrum of a Schr\"odinger operator. Notice that various results of a similar flavor were recently obtained in \cite{d2, d3}.
\begin{theorem}\label{th2}
Let $V$ be any continuous nonnegative function. Assume that $f\in
L^2(\mathbb{R}^d)$ is nonnegative and has a compact support. Let
$\sigma_f$ be the spectral measure of $f$ with respect to $H_V$ and
$\sigma'_f$ be the density of its a.c. part. Then we have
\begin{equation}
\exp\left[\frac{1}{2\pi}\int_\br \frac{\log
\sigma'_f(k^2)}{1+k^2}dk\right]\ge C_f\int f(x^0)\mathbb{E}_{x^0}
\left[ \exp\left(-\frac 12\int\limits_0^\infty V(X_\tau)
 d\tau\right)\right]dx^0 \label{th21}
\end{equation}
where the constant $C_f>0$ does not depend on $V$.
\end{theorem}
\begin{proof}
Suppose that  $V$ has a compact support.
Recall definitions \eqref{as}-\eqref{ba} from the beginning of this section.
The function $\beta_{x^0}(\theta, k)$  is analytic in $k\in
\mathbb{C}^+$ and has the the following asymptotics for large $\Im k$
\[
\beta_{x^0}(\theta,k)=1+\underline{O}(\Im k)^{-1}
\]
This follows from the analysis of the perturbation series
for the resolvent, i.e.
\[
G(x,y,k)=G_0(x,y,k)-\int G_0(x,s,k)V(s)G_0(s,y,k)ds+\ldots
\]
Notice also that the function $A_{f}(\theta,k)$, as a function in $k$,
is analytic on $\mathbb{C}^+$ and is continuous up to
$\mathbb{R}\backslash \{0\}$. The last property follows from the
limiting absorption principle, see Agmon  \cite{ag}. The representation
\begin{equation}
 u=(-\Delta-k^2)^{-1}(f-Vu) \label{rep}
\end{equation}
implies
\begin{equation}
|A_{f}(\theta,k)|\leq \frac{C_{f,V}}{{\rm dist}(k^2,\mathbb{R}_+)^{1/2}} \label{j1}
\end{equation}
for $k\to 0$.
The asymptotics of $\beta_{x^0}$ for large $\Im k$ and (\ref{ba}) yield
\begin{equation}
|A_f(\theta,k)|<\exp(C_f\Im k)(1+\underline{O}(\Im k)^{-1})
\label{j2}
\end{equation}
for $\Im k\to +\infty$. The constant $C_f$ depends on $f$ only.

From (\ref{rep}), we also have the estimate for large $k$ within the
region $\{0<\Im k<C\}$
\begin{equation}
|A_f(\theta,k)|\leq \frac{C_{f,V}}{{\rm
dist}(k^2,\mathbb{R}_+)^{1/2}} \label{j3}\end{equation} where the
constant $C_{f,V}$ depends on $V$ and $f$.

The following simple identity is true (e.g., Denisov \cite{d3})
\begin{equation}
\sigma'_{f}(k^2)=C|k|^{d-2}\|A_{f}(\theta,k)\|^2_{L^2(\Sigma_1)},
\label{fact}
\end{equation}
here $k\in\mathbb{R}, k\neq 0$. It follows from the integration by parts and the limiting absorption
principle \cite{ag}. In \cite{d3}, this formula was proved for $d=3$ but the same argument works for any $d$.

Now, observe that the function
$$
g(k)=\log \|A_f(k,\theta)\exp(ikC_f)\|_{L^2(\theta\in \Sigma_1)}
$$
is subharmonic on $\mathbb{C}^+$ (the constant $C_f$ is chosen from
(\ref{j2})). Due to the properties of $A_f$ listed above, we can
apply the mean value inequality to $g$ within the domain
$\Omega_{\epsilon, L,M}$  bounded by the curves: $\gamma_1=\{z: \Im
z=0,$ $\epsilon<|\Re z|<L\}$, $\gamma_2=\{z: \Im z>0,
|z|=\epsilon\}$, $\gamma_3=\{z: |\Re z|=L, 0<\Im z<M\}$,
$\gamma_4=\{z: \Im z=M, |\Re z|<L\}$. Letting
$\gamma=\bigcup_{j=1}^4 \gamma_j$ and $\omega_{\epsilon, L,M}$ be
the harmonic measure of $\Omega_{\epsilon, L,M}$ aiming at $i$, we
have
$$
\int\limits_{\gamma} g(k)\omega_{\epsilon, L, M}(k)d|k|\ge g(i)
$$
Taking $L\to\infty$, then $M\to\infty$, and then $\epsilon\to 0$, we
have
\begin{equation}\label{e9}
\frac{1}{\pi}\int\limits_{\mathbb{R}} \frac{\log
\|A_{f}(\theta,k)\|_{L^2(\Sigma_1)}}{1+k^2}dk\geq \log
\|A_{f}(\theta,i)\|_{L^2(\Sigma_1)}-C_f
\end{equation}
Each of these limits is justified by (\ref{j1}),(\ref{j2}), and
(\ref{j3}).

On the other hand,
\[
\|A_{f}(\theta,i)\|_{L^2(\Sigma_1)}\gtrsim \int_{\Sigma_1}
A_{f}(\theta,i)d\theta=\int_{\mathbb{R}^3} \int_{\Sigma_1}
f(x^0)a_{x^0}(\theta, i)\, d\theta dx^0
\]
as follows from \eqref{ba}. Applying the first claim of Theorem
\ref{th1} to the right hand side of the above relation and
factorization (\ref{fact}) to the left hand side of \eqref{e9}, we
come to
\begin{equation}
\exp\left[\frac{1}{2\pi}\int\limits_\br \frac{\log \sigma'_{f}(k^2)}{1+k^2}dk\right]\ge
C_f\int f(x^0)\mathbb{E}_{x^0}
\left[ \exp\left(-\frac 12 \int\limits_0^\infty
V(X_\tau)d\tau\right)\right]dx^0 \label{th2n}
\end{equation}
Now, consider any continuous nonnegative $V$ and take truncations
$V^{(n)}=V(x)\cdot \mu_n(x)$ defined in the remark after Theorem
\ref{th1}. For each $n$, (\ref{th2n}) holds true. Let
$\sigma^{(n)}_{f}$ be the spectral measure of $f$ with respect to
the Schr\"odinger operator with potential $V^{(n)}$. It is
well-known that $\sigma^{(n)}_{f}\to \sigma_f$ in the weak-star
topology. Therefore, the semicontinuity of the entropy (see
Killip-Simon \cite{ks}, for instance) applied to the left hand side
of \eqref{th2n} and the monotone convergence theorem applied to its
right hand side give (\ref{th21}).
\end{proof}

As a corollary, we get that the absolutely continuous spectrum of
$H_V$ contains $\mathbb{R}_+$ if the potential $V$ is summable over
$X_t$ with positive probability. Checking the last property for the
concrete $V$ might be difficult and should be done on a
case-by-case basis. It is conceivable that one can handle the
situation when $V$ changes sign using the technique developed in
\cite{dp} but then the statements will be generic in coupling
constant. The condition that $f$ is nonnegative is not important and
can be dropped.

The following is a simple corollary of Theorem \ref{th2},  an
estimate for the transition probability (\ref{tr-pr}) and  Jensen's
inequality.

\begin{corollary}\label{c01}
Let $V$ and $f$ be as in the previous theorem. Then,
\[
\int\limits_\br \frac{\log \sigma'_f(k^2)}{1+k^2}dk\ge-C_{1,f}\int
\frac{V(x)}{|x|^{d-1}+1}dx-C_{2,f}
\]
\end{corollary}
This result is not new; see Laptev-Naboko-Safronov \cite{lns1, lns2}
and Theorem~\ref{lns}.

Let $\Omega\subset \mathbb{R}^d$ be an unbounded domain.  We have
\begin{corollary}
Let $V$ be as in Theorem \ref{th2}, $V(x)=0$ on $\Omega$, and
$\mathbb{P}_{x^0}(X_t\in \Omega, \forall t )>0$. Then the a.c.
spectrum of  $H$ is equal to $\mathbb{R}_+$. Moreover,
\begin{equation}
a_{x^0}(\theta)\gtrsim \mathbb{P}_{x^0}(G_t\in \Omega, \forall
t)\label{lapa}
\end{equation}
where $G_t$ is from  (\ref{fk}).
\end{corollary}

\section{The modified  harmonic measure, capacity and their properties}\label{s3}
The results of the previous section suggest that one needs to study
functions $a_{x^0}$ in different directions. One can use
probabilistic or analytic methods for this purpose. We are going to
consider an important case when $V=0$ on $\Omega$ and one has no
control on the size of $V$  on $E=\Omega^c$. We only handle  $d=3$
in this section and explain  how the results we obtain look like for
$d=2$. The situation when $d>3$ can be treated similarly and will be
discussed briefly in the Appendix B. Without loss of generality we
can always assume that $\theta=\pm e^1$.

Let $\Omega\subset \br^3$ be a given domain with the compact
$E=\Omega^c$. Take $\Gamma=\partial \Omega$ and consider the
Dirichlet problem
\begin{equation}
\frac 12 \Delta u+u_{x_1}=0, \quad u|_\Gamma=f \label{e95}
\end{equation}
where $u$ decays at infinity, $f\in C(\Gamma)$. The solutions of
Dirichlet problem for general elliptic equations on the arbitrary
bounded domains go back to P\"uschel \cite{puschel} and Oleinik \cite{oleinik}.

The domains $\Omega$ we are interested in are unbounded but the same results easily
follow by an approximation argument. Furthermore,  it is
proved in \cite{puschel, oleinik} that the regular points for (\ref{e95}) coincide with the
Wiener regular points (i.e., regular points w.r.to $\Delta$) and thus can be identified by the Wiener's test \cite[Sect. III.7]{GM}.  That said, one can fix any reference point $x^0\in \Omega$
and consider the solution of (\ref{e95}) as a linear bounded
functional $\Phi_{x^0}$ of $f$. By the Riesz
representation theorem, we have
$$
u(x^0)=\Phi_{x^0}(f)=\int_\Gamma f d\omega_{x^0}
$$
The measure  $\omega(x^0, A, \Omega), \ A\subseteq \Gamma$ is called
the modified harmonic measure of $A$. Clearly, any regular point of
$\Gamma$ belongs to the support of the modified harmonic measure and
the set of irregular points has capacity zero, see Mizuta \cite[ p.
136]{mizuta}.

We will be mostly interested in the
estimates on $\omega(. ,\Gamma, \Omega)$, which is the solution of
(\ref{e95}) with $f=\chi_A, A\subseteq \Gamma$. It is known \cite{oleinik} that there exist
domains  $\Omega_n$ with piece-wise smooth boundaries such that
$\overline{\Omega}_n\subset \Omega_{n+1}$, $\bigcup\limits_{n}\Omega_n=\Omega$,  and
\begin{equation}
\omega(x^0,\Gamma,\Omega)=\lim_{n\to\infty}
\omega(x^0,\Gamma_n,\Omega_n), \quad \Gamma_n=\partial\Omega_n
\label{appr3}
\end{equation}
This approximation allows us to assume the smoothness of the
boundary later on.

The probabilistic meaning of $\omega$ is
\[
\omega(x^0,A,\Omega)=\mathbb{P}_{x^0}(G_t \ {\rm hits} \ \Gamma
 \ {\rm for \ the \ first \ time \ at } \  A  )
\]
where $G_t=x^0+t\cdot e^1+B_t, x^0\in\Omega$. If one solves
\[
-\frac 12 \Delta \psi=-\frac 12\psi, \ x\in \Omega, \quad
\psi|_\Gamma=\chi_Ae^{x_1},\quad
\]
then
\[
\omega(x)=\psi(x) e^{-x_1}
\]
Thus, assuming that $\Omega$ has a piece-wise smooth boundary, the
measure $\omega$ has piece-wise smooth density as well and it is
given by
\[
-e^{-x_1+\xi_1}\frac{\partial}{\partial n_\xi} G(x; \xi),
\]
where  $G$ is the Green's function for $2(-\Delta_D+1)^{-1}$, $\Delta_D$ is the Laplacian with Dirichlet boundary condition on $\Gamma$,  and $\xi=(\xi_1,\xi')\in \Gamma$.

The estimate from below for  $1-\omega(x^0,\Gamma,\Omega)$ is what
we need to control in (\ref{lapa}) in order to guarantee that the
asymptotics of Green's function in the direction $e^1$ is
comparable to the asymptotics of the free Green's function. The
similar problem of visibility of infinity from the origin (for the
standard Brownian motion) was recently considered in Carroll--Ortega-Cerd\`a \cite{J}.

In this section, we will also build the potential theory for the
operators
\[
L^\pm=L^{\pm e^1}=\frac 12 \Delta \mp \partial_{x_1}
\]
A particular attention is paid to the study of the corresponding
capacity and harmonic measure.  The definitions and ideas of the
proofs are close in the spirit to the constructions from the potential theory for  the
elliptic and parabolic cases,  e.g., Doob \cite{doo}. From this point of view, many results
are rather standard. In the meantime, we feel like we have to write them up for the reader's
convenience. Only Theorem~\ref{proj} is substantially new and seems
not to be known even for the parabolic capacities to the best of our
knowledge.

Through the rest of the paper, the prefix ``$L^\pm$" in front of
adjectives will be systematically dropped, i.e.,  $L^\pm$-harmonic
functions will be called just harmonic functions etc.

\subsection{Potential theory: a special case}\label{ss31}
The content of this subsection follows the lines of the general potential theory as presented in Landkof \cite{lan}, Hayman-Kennedy \cite{hk} and Garnett-Marshall \cite{GM}.

Let $z=(z_1,z'), \xi=(\xi_1,\xi')\in\br^3$, and the reference direction be $+e^1$. We introduce the potential
\begin{equation}\label{e9501}
K^-(z,\xi)=2G_0(z,\xi) e^{\xi_1-z_1}
\end{equation}
 where $G_0$ is Green's function for $(-\Delta+1)^{-1}$ on
$\mathbb{R}^3$, that is
\[
G_0(z,\xi)=G_0(z,\xi; i)=\frac 1{4\pi}\frac{e^{-|z-\xi|}}{|z-\xi|}
\]
Obviously, $G_0(z,0)\sim \frac 1{4\pi}|z|^{-1}$ as $|z|\to 0$.
So, for small $|z-\xi|$,  the potential $K(z,\xi)$ behaves like the
standard elliptic potential for $\br^3$ and for $|z-\xi|$ large it is
similar to the parabolic potential. Thus, we expect two regimes: the
microscopic one will mimic the elliptic theory and macroscopic will
have some resemblance to the parabolic case.

For  the ``dual" reference direction $-e^1$, the differential operator is $L^+$ and
$$
K^+(z,\xi)=2G_0(z,\xi) e^{z_1-\xi_1}
$$
It is important that
\begin{equation}
K^-(z,\xi)=K^+(\xi, z)\label{refl}
\end{equation}

Let $E$ be a compact subset of  $\br^3$ and $\cp(E)$ be the set of probability Borel measures on $E$. For $\mu\in\cp(E)$, put
\[
U_{\mu}^\pm(z)=\int_E K^\pm(z,\xi)d\mu(\xi)
\]
to be the corresponding potential. Clearly, $U_\mu^\pm$ is lower
semicontinuous. Consider
\begin{equation}
C^\pm(E)=\lt(\inf_{\mu\in \cp(E)} \sup_{z\in \mathbb{R}^3}
U_{\mu}^\pm(z)\rt)^{-1} \label{def1}
\end{equation}

\begin{definition}\label{d1} We call $C^\pm(E)$ the modified capacity
of a set $E$ in the direction $\mp e^1$.
\end{definition}
It is clear that the capacity is translation-invariant but is not invariant under the rotation, in
general. Since $E$ is a compact, $C^{\pm}(E)=0$ if and only if the elliptic capacity of $E$ is zero as well (i.e., the polar sets in our case are the same as in the standard elliptic theory).
Below, we mostly discuss the ``$-$"-case; the ``$+$"-case can be handled similarly.

The capacity can also be defined in the following way. Introduce the class of admissible measures $\ca^-(E)$ as follows:
$\nu\in \ca^-(E)$ iff $\nu$ is positive measure supported on $E$ and
\begin{equation*}
\sup_{z\in \mathbb{R}^3} U^-_{\nu}(z)=1
\end{equation*}
Then, we have
\begin{equation}
\sup_{\nu\in \ca^-(E)}\nu(E)=C^-(E) \label{second}
\end{equation}
Assume $0<C^-(E)<\infty$ and let  $\{\nu_n\}$ be a maximizing
sequence  to \eqref{second}. We denote one of  its weak limits by
$\nu^-$. Then, $\nu^-(E)=C^-(E)$. We have (\cite{GM}, Lemma 4.2)
\[
U_{\nu^-}^-(z)\leq \liminf_{n\to\infty} U_{\nu_n}^-(z)\leq 1
\]
and
\[
\alpha=\sup_{z\in\mathbb{R}^3} U_{\nu^-}^-(z)\leq 1
\]
Thus, $\hat\nu=\nu^-/\alpha\in \ca^-(E)$ and
$\hat\nu(E)=C^-(E)/\alpha\geq C^-(E)$ which means that $\alpha=1$ and
$\nu^-$ is a maximizer. Therefore, a minimizer $\mu^-$ for
(\ref{def1}) exists and is equal to $\nu^-/C^-(E)$.

The main results of the elliptic theory are true in our case as well
and we list some of them below for the reader's convenience. They
are stated for $C^-$ but their analogs hold with respect to any
direction.
\begin{itemize}
\item[1.] Monotonicity:
for  $E_1\subseteq E_2$, then $C^-(E_1)\leq C^-(E_2)$.

\item[2.] Subadditivity:
if $E= \bigcup\limits_j E_j$, $E_j$ are disjoint, then
\[
C^-(E)\leq \sum_j C^-(E_j)
\]
Indeed, we use (\ref{second}). If $\nu^-$ is a maximizer in
(\ref{second}) for $E$ and $\nu_j$ is its restriction to $E_j$,
then
\[
\nu_j(E_j)\leq C^-(E_j) \sup_z U^-_{\nu_j} (z)\leq C^-(E_j)\sup_z
U^-_{\nu^-}(z)=C^-(E_j)
\]
Hence
\[
C^-(E)=\nu^-(E)=\sum_j \nu_j(E_j)\leq \sum_j C^-(E_j)
\]

\item[3.] Macroscopic scale:
let $T_h=[0,h^2]\times \Pi'_h, \ h>1$, and $\Pi'_h=[0,h]^2$. Then
\[
C^{-}(T_h)\sim h^2
\]
Indeed, take $z_h=-h^2 e^1$. Then, for any $\mu\in
\cp(T_h)$, we have
\[
U^-_{\mu}(z_h)\sim \frac 1{h^2}
\]
thus
\[
\inf_{\mu}\sup_z U^-_\mu(z)\ge \frac{C_1}{h^2}
\]
On the other hand, if $\mu$ is the normalized Lebesgue measure on $T_h$, we see
\[
\sup_z U^-(z)\leq \frac{C_2}{h^2}
\]
and thus
\[
\inf_{\mu} \sup_z U^-_{\mu}(z)\leq \frac{C_2}{h^2}
\]

\item[4.] Microscopic scale:
the following proposition is immediate from the properties of the
kernel $K^-$ and the definition of the standard Wiener capacity
$C_W(E)$ (i.e., the one related to $\Delta$):
\begin{proposition}\label{p10}
If ${\rm diam}(E)\lesssim 1$, then $C^{-}(E)\sim C_W(E)$.
\end{proposition}
So, for example,  $C^{-}(B(z,r))\sim r$, as $r\to 0$.\smallskip

\item[5.] If $G_n$ is a sequence of compacts, $G_{n+1}\subseteq
G_{n}, \bigcap\limits_n G_n=E$, then
\begin{equation}
\lim_{n\to\infty} C^{-}(G_n)=C^{-}(E) \label{lim2}
\end{equation}

Indeed, assume that $\nu_n$ is a maximizer in $(\ref{second})$ for
$G_n$. Then, one can find $\nu_{k_n}$ that converges weakly to $\nu$
supported on $E$. Since $U^-_{\nu_n}(x)\leq 1$, we have
$U^-_{\nu}(x)\leq 1$ as well. Thus, $C^-(E)\geq \nu(E)$. On the
other hand, $\nu(E)=\lim\limits_n \nu_{k_n}(G_{k_n})=\lim\limits_n
C^-(G_{k_n})\geq C^-(E)$ since $C^-(G_n)\geq C^-(E)$ by
monotonicity.

\end{itemize}

The last approximation result with $G_n$ having the piece-wise
smooth boundary allows one, like for modified harmonic measure, to
reduce the analysis to the smooth case.

At last, let $\theta$ be a fixed vector from $\Sigma_1$ and $x=\theta r, |x|=r$. For any measure $\mu$, we have
\begin{eqnarray}
&&\lim_{r\to+\infty}r U_{\mu}^-(\theta r)=0, \quad  \theta\neq-e^1, \nonumber\\
&&\lim_{r\to+\infty} 2\pi r\,  U^-_{\mu}(-e^1 r)=\mu(E) \label{second1}
\end{eqnarray}

\subsection{More on modified harmonic measure and capacity}\label{ss32}

In this subsection, we will relate the capacity to the modified
harmonic measure of a compact. We first assume that $E$ has a
piece-wise smooth boundary, e.g., it is a finite union of closed
balls. Then the approximation results (\ref{appr3}) and (\ref{lim2})
will enable us to handle the case of any compact.

Take a compact $E$  in the half-space $\{x =(x_1,x'): x_1>1\}$.
Let $\alpha_z$ be the density of the harmonic measure  for the
half-space $\Pi_-=\{x: x_1<0\}$ and the reference point $z\in
\Pi_-$, e.g.,
\[
\mathbb{P}_z (G_t {\rm \ first \ hits \ the \ plane}
\ \{x: x_1=0\} {\rm \ at \ set \ } B)=\int_B \alpha_z(y')dy'
\]
where $B\subset \{x: x_1=0\}$. Then,
\[
\omega^-(z, A,\Omega)=\int_{\{x_1=0\}}
\alpha_z(y')\omega^-(y',A,\Omega) dy'
\]
for any $A\subset \Gamma$. Then, introduce the sweeping of the two-dimensional Lebesgue measure on $\{x_1=0\}$ to $\Gamma$,
\[
p^-(A)=\int_{\{x_1=0\}} \omega^-(y', A,\Omega)dy'
\]
The explicit formula for $\alpha_z$ is easy to write down, namely
$$
\alpha_z(y')=\left.-\frac 1{2\pi}  e^{-z_1} \frac{\pt}{\pt
y_1}\lt(\frac{e^{-|z-y|}}{|z-y|}-\frac{e^{-|z_*-y|}}{|z_*-y|}\rt)\right|_{y_1=0}
$$
where $z_*$ is symmetric to $z$ with respect to $\{y: y_1=0\}$. The
above formulae imply  that
\[
\omega^-((-r,y'), A,\Omega) \approx \frac{1}{2\pi r}\,  p^-(A)
\]
as $r\to+\infty$ and $y'$ is fixed. Notice that, due to our assumptions on $\Gamma$,
 $p^-$ is a measure on $\Gamma$ with piece-wise smooth density. It can be interpreted as follows:
if $L^-u=0$ on $\Omega$, $u\in C(\overline{\Omega})$, and decays at
infinity, then
\begin{equation}\label{e91}
\lim_{r\to+\infty} (2\pi r)\, u(-r,y')=\int_\Gamma u(\xi)dp^-(\xi)
\end{equation}
Consider  the following potential
\[
U^+_{p^-}(z)=\int_E K^-(\xi,z)d p^-(\xi)\, \lt(=\int_E K^+(z,\xi)d p^-(\xi)\rt)
\]
If $z$ belongs to the interior of $E$, then $f(\xi)=K^-(\xi,z)$ satisfies $L^-f=0$ outside
$E$ and therefore, referring to \eqref{e91},
\begin{equation}\label{e92}
U^+_{p^-}(z)=\lim_{r\to+\infty} (2\pi r)\, K^-((-r,0), z)=1
\end{equation}
which holds on the interior of $E$ and, by continuity, on all of
$E$. By the maximum principle for the operator $L^+= \frac 12
\Delta-\partial_{x_1}$ \cite[Ch. 6]{lan}, one has $\sup\limits_{z\in
\mathbb{R}^3} U^+_{p^-}(z)=1$, and so
\begin{equation}
C^{+}(E)\geq p^{-}(E)\label{i1}
\end{equation}
Similarly,
\begin{equation}
C^{-}(E)\geq p^{+}(E)\label{i2}
\end{equation}
and, for any $z\in E$
\[
U^-_{p^+}(z)=1, \quad  U^+_{p^-}(z)=1
\]
The duality argument then gives
\[
p^+(E)=\int_E \left(\int_E K^-(\xi,z)dp^-(\xi)\right) dp^+(z)=p^-(E)
\]
by Fubini. Moreover, one has
\begin{eqnarray*}
p^+(E)&\geq& \int_E \left( \int_E K^-(\xi,z)d\nu^+(\xi)\right)
dp^+(z)= \int_E \left( \int_E K^-(\xi,z)dp^+(z)\right) d\nu^+(\xi)\\
&=&\nu^+(E)=C^+(E)
\end{eqnarray*}
Therefore, for $E$ with piece-wise smooth boundary, we have
\begin{equation}
C^+(E)=C^-(E)=p^+(E)=p^-(E) \label{ravenstvo}
\end{equation}
and $p^+$ is a minimizer in $(\ref{second})$. Moreover, $U^-_{p^+}(x)$ is solution to the
 $L^-$--Dirichlet problem in $\Omega$ with boundary value $f=1$.

Now, consider an arbitrary compact $E$. The measures $p^{\pm}$ can be defined in the same way, they are both supported on $\Gamma$. Then, we have

\begin{theorem}\label{th_trii}\hfill
\begin{itemize}
\item[\it i) \ ] For any compact $E$
\[
C^+(E)=C^-(E)=p^+(E)=p^-(E)
\]
\item[\it ii)]  $p^+$ is a maximizer for (\ref{second}) and $U^-_{p^+}=1$ on
the interior of $E$. Moreover, $U^-_{p^+}$ is the generalized
solution to the  Dirichlet problem
\begin{equation}\label{e1001}
L^- U^-_{p^+}=0,\quad U^-_{p^+}\big|_{\Gamma}=1
\end{equation}
the boundary condition being understood quasi-everywhere.
\end{itemize}
\end{theorem}

\begin{proof}
Consider a sequence of compacts $G_n$ with piece-wise smooth
boundaries that decreases to $E$ as in \eqref{lim2}. Obviously,
$p^+_{n}$ converges weakly to $p^+$. Moreover, $U^-_{p^+_n}(x)=1$ on
$E$ and thus $U^-_{p^+}(x)=1$ for any interior point  $x\in E$ (if
there is any). Since $U^-_{p^+_n}$ converges to $U^-_{p^+}$ on
$\Omega$ and each $U^-_{p^+_n}$ solves the Dirichlet problem for
$G_n$ with boundary value $f=1$, we get that $U^-_{p^+}$ is
generalized solution to \eqref{e1001}. Since $U^-_{p^+}$ is also
quasicontinuous on $\mathbb{R}^3$, we get $U^-_{p^+}=1$
quasi-everywhere on $\Gamma$. Since  we have $U^-_{p^+}\leq 1$ on
$\br^3$, $p^+$ is admissible. On the other hand, by (\ref{lim2}) and
(\ref{appr3}), we get $C^-(E)=\lim\limits_{n}
C^-(G_n)=\lim\limits_{n} p^+(G_n)=p^+(E)$. This shows that $p^+$ is
a maximizer.
\end{proof}
From now on, the symbol $C(E)$ will denote  $C^-(E)=C^+(E)$.

The proof of the forthcoming uniqueness result relies on the
following extended minimum principle.

\begin{theorem} Let $O$ is bounded domain and $u$ is bounded
$L^-$--harmonic function on $O$. If $S$ is polar subset of $\partial
O$ and $\liminf_{x\to\xi} u(x)\geq M$ for any $\xi\in \partial O, \xi\notin S$, then
either $u(x)>M$ or $u(x)=M$ in $O$.\label{gmp}
\end{theorem}

The proof of this theorem is identical to the proof of
 Theorem 5.16 of \cite[Ch. 5]{hk} as long as one has the following

\begin{lemma}\label{l101} Let $S$ be a bounded polar set. Then there is a nonnegative
 $L^-$-superharmonic function $u$ on $\mathbb{R}^3$ such that
$u(x)=+\infty$ for any $x\in S$ and $u(x^0)$ is finite at a
prescribed point $x^0\notin S$.
\end{lemma}
\begin{proof}
By \cite{hk}, Theorem 5.11,  and the Riesz representation theorem for
superharmonic functions, there exists a compactly supported measure $\nu$ on
$\mathbb{R}^3$  such that for the potential
$\psi(x)=\int |x-y|^{-1}d\nu(y)$ equals $+\infty$ on $S$
and is finite at $x^0$. Consider the function
\[
u(x)=\int K^-(x,y)d\nu(y)=C\psi(x)+\int \tilde{K}(x,y)d\nu(y)
\]
where $\tilde{K}(. , . )$ is a continuous kernel. Clearly, $u$ has the properties claimed in the lemma.
\end{proof}

\begin{theorem} The measure $p^+$ is the unique maximizer in (\ref{second}).
\end{theorem}

\begin{proof}
Assume first that $E$ does not have interior points. Take again
$G_n$ as in the proof of Theorem \ref{th_trii}. Let $\ti p$ be
another maximizer in (\ref{second}). From Theorem \ref{th_trii}, we
have
\[
\int dp_n^-(x)\int K^-(x,y)d\ti p(y)=\int d\ti p(y)\int
K^+(y,x)dp_n^-(x)=C(E)=C(G_n)-\epsilon_n
\]
where
\[
\epsilon_n=C(G_n)-C(E)\to 0
\]
The function
\[
h_n=U_{p_n^+}^- -U_{\ti p}^-
\]
is $L^-$-harmonic on $G_n^c$ and nonnegative there by the maximum
principle since $U_{\ti p}^-(x)\leq 1, \ U_{p_n^+}^-(x)=1$ on
$\Gamma_n=\partial G_n$ and they both decay at infinity. Moreover,
\[
\int h_n(x) dp_n^-(x)=\epsilon_n\to 0
\]
Therefore,
\[
\int_{\{x_1=0\}} h_n(x)dx'=\epsilon_n\to 0
\]
and by Harnack principle, $h_n(x)\to 0$ on any compact in $\Omega$.
Since $U^-_{p_n^+}(x)\to U^-_{p^+}(x)$ on $\Omega$, we have
$U^-_{\ti p}(x)=U^-_{p^+}(x)$ on $\Omega$. This means
that $U^-_{\ti p}(x)=U^-_{p^+}(x)=1$ quasi-everywhere on $E$ since the
potential $U^-_{\ti p}$ is quasicontinuous on $\mathbb{R}^3$ \cite[p. 73]{mizuta}. Then, $U^-_{\ti p}=U^-_{p^+}$ a.e. on $\br^3$. Taking Fourier transform,  we get
\[
(w^2+1)^{-1}\cal{F}(e^{x_1}\ti p)=(w^2+1)^{-1}\cal{F}(e^{x_1}p^+)
\]
and thus the measures are equal.

Now, consider the general case when $E$ has nonempty interior.
Consider any open component in the interior and call it $O$. The
boundary $\Gamma'=\partial O$ is a subset of $\Gamma$. In a similar
way, we can prove that  the potential $u=U^-_{\ti p}$ is continuous
and equals to $1$ quasi-everywhere on $\Gamma'$. We also know that
$u$ is $L^-$-superharmonic on $O$ and $0\leq u(x)\leq 1$. Let us apply
now the Theorem~\ref{gmp} to $O$ and $u(x)$ with $M=1$ and $S$ taken
as a set of points of $\Gamma'$ at which $U^-_{\tilde{p}}$ is not
continuous or its value $\neq 1$. Then, $U^-_{\tilde{p}}=1$ on $O$
and thus $\ti p=p^+$ by the same argument.
\end{proof}
In view of these results, we can call $p^{\pm}$ the equilibrium
measures.\bigskip

Let $\bi=(i_2,i_3)\in \bz^2$.
For a given $\bi$, $z\in\br^3$ and $T>0$, put $\Pi'_\bi=[i_2T, (i_2+1)T]\times[i_3T,(i_3+1)T]$ if $i_{2(3)}\geq 0$.
The following proposition provides another useful relation between the capacity
and harmonic measure.

\begin{proposition}\label{trii}
Assume that $T$ is large and $E\subset [T^2, 2T^2]\times \Pi'_\bi$. Then
\[
\omega^-(0,E,\Omega)\sim \frac{C(E)}{T^2}
\]
for small $|\bi|$ and
\[
\omega^-(0,E,\Omega)\lesssim \exp\left( -\frac{|\bi|^2}{2+\sqrt{4+|\bi|^2T^{-2}}}\right) \frac{C(E)}{T^2}
\]
in general.
\end{proposition}
\begin{proof}
We can always assume that $\partial E$ is smooth. Then, consider
\[
U^-(z)=\int K^-(z,\xi)dp^+(\xi)
\]
We know that $L^- U^-=0$ on $\Omega$ and $U^-=1$ on $\partial E$. Therefore,
\[
\omega^-(0,E,\Omega)=U^-(0)\sim \frac {p^+(E)}{T^2}=\frac{C(E)}{T^2}
\]
for $|\bi| < i_0$ and a simple computation yields
\begin{equation}
\omega^-(0,E,\Omega)\lesssim \exp\left( -\frac{|\bi|^2}{2+\sqrt{4+|\bi|^2T^{-2}}}\right)
\frac{C(E)}{T^2}\label{e6}
\end{equation}
\end{proof}
As the law of  iterated logarithm suggests, the interesting range
for $|\bi|$ is $|\bi|\leq C\sqrt{T\log\log T}$ with some $C$. Then,
in (\ref{e6}), the weight is dominated by $\exp(-(\frac 14
-\epsilon)|\bi|^2)$, and $\epsilon>0$ is arbitrary.\bigskip

In some cases, the capacity is ``almost additive".
\begin{proposition}\label{addup}
Let $E=\bigcup_j E_j$ such that  disjoint sets $E_j\subset
[0,T^2]\times \Pi'_{\bi_j}, \bi_j\in\bii$, for some $T>1$, and
$\bii$ be finite. Then
\[
C(E)\gtrsim \sum_j C(E_j)
\]
\end{proposition}
\begin{proof}
Let $\nu^-_j$ be a maximizer for $E_j$. Take $\nu=\sum \nu^-_j$.
Then,
\[
\nu(E)=\sum C(E_j)
\]
On the other hand, we can estimate the potential in the following
way. Take any point $z$ in, say, $E_1$. Then
\[
U^-_\nu(z)=\sum_j U^-_{\nu^-_j}(z)
\]
For $|\bi_j-\bi_1|\leq 2$ we use estimates $U^-_{\nu^-_j}(z)\leq 1$.
For any $\bi_j, |\bi_j-\bi_1|\geq 3$ and $z_{\bi_j}=z+(\bi_j-\bi_1)
T$, we have
\[
K(z,\xi)\leq K(z_{\bi_j},\xi) e^{-\alpha |\bi_j-\bi_1|},
\]
for any $\xi\in E_j$ and some universal $\alpha>0$. Therefore,
\[
U^-_{\nu^-_j}(z)\leq e^{-\alpha {|\bi_j-\bi_1|}}
\]
and so
\[
U^-_{\nu}(z)\leq C
\]
This implies the required bound.
\end{proof}

\subsection{Anisotropic Hausdorff content and capacity}\label{ss33}
Since we have to consider both microscopic and macroscopic regimes,
one has to define the anisotropic Hausdorff content in the following
way \cite[App. D]{GM}.

Let $h: \br_+\to \br_+$ be some measure function, i.e. $ h(0)=0$,
$h$ is continuous and increasing. Set $\Pi'_r=[0,r]\times[0,r]$ for
$r\ge0$. We say that $H(r)$ is the characteristic set if it can be
translated to a parallelepiped $[0,r^2]\times \Pi'_r$ if $r\ge 1$
and to a cube $[0,r]\times \Pi'_r$ if $r<1$. We say that the size of
$H(r)$ is $r$. Cover $E$ by a countable number of $\{H(r_j)\}$. Then
the Hausdorff content of $E$ (w.r.to $h$) is
\begin{equation}\label{e9502}
M_h(E)=\inf \sum_j h(r_j)
\end{equation}
where the infimum is taken over all coverings.

Take
\begin{equation}
h(r)=\left\{
\begin{array}{cc}
r,& 0<r\leq 1\\
r^2,& r>1
\end{array}\right.\label{mes}
\end{equation}
Since the capacity is subadditive, one has a simple

\begin{proposition}\label{content}
For any compact set $E$ and the measure function $h$ given by
(\ref{mes}), we have
\[
C(E)\lesssim M_h(E)
\]
\end{proposition}
One can compare this to \cite[App. D]{GM}, Theorem D.1. This is a
key estimate which will allow us to guarantee the wave propagation
(see Section~\ref{s5}) in terms of the metric properties of the set
of obstacles $E$. It seems possible to relate the capacity to the
Hausdorff dimension or other metric properties of the set similarly
to the elliptic/parabolic case (see, e.g. an excellent paper by
Taylor-Watson \cite{tw}). We do not pursue this here.

\medskip\nt
{\bf Remark.} If $d=2$, we need to introduce the measure function
differently, i.e.
\begin{equation}
h(r)=\left\{
\begin{array}{cc}
|\log r|^{-1},& 0<r\leq 1/2\\
Cr,& r>1/2
\end{array}\right., \quad C=\frac{2}{\log 2}
\end{equation}
The proposition above and the results below then hold true.

\subsection{How does the capacity change under projection?}\label{s34}
It is a well known fact in classical  potential theory  that the capacity
(Wiener or logarithmic) can only decrease when the set is contracted.
Analogous statements for harmonic measure are known as Beurling's projection theorem and
Hall's lemma, see \cite[Sect. 3.9]{GM}. We are interested in proving similar results for
modified capacity. Naturally, the results will be anisotropic.

If diam$(E)$ is small then all results for the Wiener's capacity
hold. It is the macroscopic regime that we are more interested in.
Let $\ppr'$ be the projection on the plane $\{x_1=0\}$ and $\aleph$
be any contraction in the $x'$-plane (i.e., the plane $OX_2X_3$).
That is, under the action of $\aleph$, the $x_1$ coordinates of
points in $E$ stay the same, but the $x'$-sections of $E$ are
contracted. This makes the capacity of the set smaller, compare this
to \cite[Ch. 3]{GM}, Theorem~4.5.
\begin{theorem}\label{projx}
For any set $E$, we have
\begin{equation}
C({\aleph} E)\leq C(E)  \label{prx}
\end{equation}
\end{theorem}
\begin{proof}
We can assume that $E$ has piece-wise smooth boundary. Consider any measure $\nu$ on $E$ and take its projection under contraction, call it $\tilde\nu$.
\[
\tilde\nu(\aleph E)=\nu(E)
\]
We also have
\[
K^-(z,\xi)\leq K^-(\aleph z, \aleph \xi)
\]
for any $z,\xi \in E$ and so
\[
U^-_{\nu}(z)\leq U^-_{\tilde\nu}(\aleph z)
\]
Now, take $\tilde\nu$ to be the maximizer for $C^-(\aleph E)$ and let $\nu$
be its preimage under $\aleph$ such that $\nu(E)=\tilde\nu(\aleph E)$.  Then,
\[
U^-_{\nu}(z)\leq 1
\]
so $C(E)\geq C(\aleph E)$ by definition.
\end{proof}

An interesting question is what happens under the action of $\ppr'$.
We are interested in the case when $E$ consists of the finite number
of disjoint components with piece-wise smooth boundaries, e.g., a
finite union of balls or parallelepipeds. The notation $|\ppr' E|$
stands for Lebesgue measure of the projection of $E$. Let $\ccc$ be
the set of all contractions in the plane $\{x_1=0\}$ that preserve
the Lebesgue measure of $\ppr' E$.  Define
\begin{equation}\label{e93}
\ti D(\ppr' E)=\inf_{\aleph\in\ccc} \diam \aleph(\ppr' E)
\end{equation}

\begin{theorem}\label{proj}
Let $H_T=[0,T^2]\times \Pi'_T, \, T>1$, $E\subset
H_T$, and $\ti D(\ppr' E)>2$. Then
\begin{equation} \label{po}
\frac{|\ppr' E|}{\log \ti D(\ppr' E)} \lesssim C(E)
\end{equation}
This inequality is sharp, at least for sets $E$ with the property $\ti D(\ppr' E)\sim \diam\ppr' E$.
\end{theorem}

\begin{proof}
We start with a certain reduction.  By the standard approximation
argument, one can focus on the case when $E=\bigcup^N_{j=1} E_j$,
where $E_j\subset\{x: x_1=x^j_1\}$ is a (planar) compact and
different $\ppr' E_j$ are disjoint.

\medskip
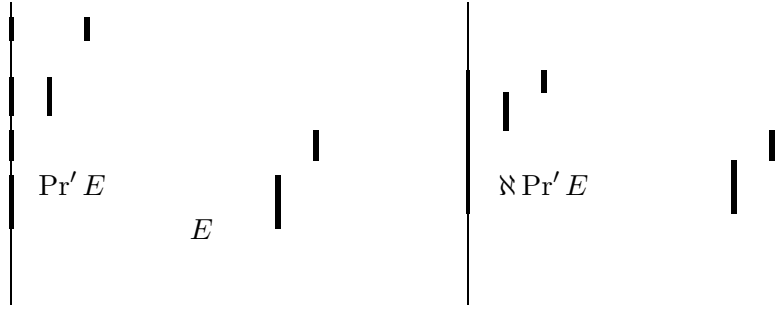
\begin{figure}

\ifx\JPicScale\undefined\def\JPicScale{1}\fi \unitlength
\JPicScale mm

\begin{center}

\begin{picture}(200,40)(-20,0)

\linethickness{0.6mm} \put(65,23){\line(0,1){5}}
\linethickness{0.6mm} \put(70,28){\line(0,1){3}}
\linethickness{0.6mm} \put(95,12){\line(0,1){7}}
\linethickness{0.6mm} \put(100,19){\line(0,1){4}}

\linethickness{0.1mm} \put(60,0){\line(0,1){40}}
\linethickness{0.4mm}  \put(60,12){\line(0,1){19}}

\linethickness{0.6mm} \put(5,25){\line(0,1){5}}
\linethickness{0.6mm} \put(10,35){\line(0,1){3}}
\linethickness{0.6mm} \put(35,10){\line(0,1){7}}
\linethickness{0.6mm} \put(40,19){\line(0,1){4}}

\linethickness{0.1mm} \put(0,0){\line(0,1){40}}

\linethickness{0.4mm} \put(0,25){\line(0,1){5}}
\linethickness{0.4mm} \put(0,35){\line(0,1){3}}
\linethickness{0.4mm} \put(0,10){\line(0,1){7}}
\linethickness{0.4mm} \put(0,19){\line(0,1){4}}

\put(25,10){\makebox(0,0)[cc]{$E$}}
\put(8,16){\makebox(0,0)[cc]{$\ppr' E$}}
\put(70,16){\makebox(0,0)[cc]{$\aleph \ppr' E$}}
\end{picture}

\end{center}
\caption{Action of contraction $\aleph$ on $\ppr' E$ ($d=2$).}\label{ff1}
\end{figure}

We also suppose that $\diam (\ppr' E)\lesssim \ti D(\ppr' E)$ since otherwise we can use the previous theorem.
So, let $\upsilon=m_2|_E$, where $m_2$ is the (planar) Lebesgue measure.  Of course,
$\upsilon(E)=|\ppr' E|$.
Fix any $z=(z_1,z')\in E$. That is, $z_1=x^{(j_0)}_1$ for a $j_0\in \{1, \dots, N\}$. For brevity, put $E'$ ($E'_j$) to be the $x'$-projection of $E$ ($E_j$, respectively) on $\{x: x_1=x^{(j_0)}_1\}$ and $\ti d=\ti D(\ppr' E)$. Notice that, up to a translation in the $x'$-direction,  $E'\subset B'(0, \ti d)$.
Then,
\[
U^-_{\upsilon}(z)\lesssim 1+\int_{B'(0,\ti d)\bsl B'(0,1)}
\chi_{E'}(x')
\frac{e^{-(|x'|^2+h^2(x'))^{1/2}+h(x')}}{(|x'|^2+h^2(x'))^{1/2}}\,
dx' = 1+I
\]
where $h: E'\to \br_+$ is the step function parameterizing the set
$E$, i.e. $h(x')=x^{(j)}_1-x^{(j_0)}_1$ for $x'\in \ppr' E_j$. In
$I$, the contribution from the set where $h(x')\leq 0$ is controlled
by the absolute constant. Then,
\begin{eqnarray*}
I&\leq&C+\int_{B'(0,\ti d)\bsl B'(0,1), h(x')>0}
\frac{e^{-(|x'|^2+h^2(x'))^{1/2}+h(x')}}{(|x'|^2+h^2(x'))^{1/2}}\chi_{E'}(x') dx'\\
&\leq& C+\int_{B'(0,\ti d)\bsl B'(0,1), h(x')>0} \frac{e^{-\frac12
\frac{|x'|^2}{(|x'|^2+h^2(x'))^{1/2}}}}
{(|x'|^2+h^2(x'))^{1/2}}\chi_{E'}(x') dx'
\end{eqnarray*}
Let
\[
\phi(x')=\frac 12 \frac{|x'|^2}{(|x'|^2+h^2(x'))^{1/2}}
\]
and $Z_n=\{x': n<4\phi(x')\leq n+1\}, \quad n=1,2,\ldots$. Then,
\begin{eqnarray*}
I&\lesssim&1+\int_{x'\in B'(0,\ti d)\bsl B'(0,1), \ |x'|^2\leq
h(x')}(\dots)+\sum_{n}
\int_{x'\in B'(0,\ti d)\bsl B'(0,1), \ x'\in Z_n} (\dots)\\
&\lesssim& 1+\int_{B'(0,\ti d)\bsl B'(0,1)}
\frac{1}{|x'|^2}dx'+\sum_{n} \int_{B'(0,\ti d)\bsl B'(0,1)} e^{-n/4}
\frac {n}{|x'|^2}\,dx'\lesssim 1+\log \ti d
\end{eqnarray*}
This estimate and the maximum principle prove inequality (\ref{po}).

It is more interesting that this estimate is essentially sharp under assumptions mentioned in the statement of the theorem.

Let $T=2^N$. We will take the set $E_T=A_T\times [0,T]$ where $A_T$
is the subset of $\tilde H_T=[0, T^2]\times [0,T]$ constructed as
follows. The set $A_T$ will be chosen such that $E_T'=[0,T]\times
[0,T]$. Divide $\tilde H_T$ into eight equal parts each
translationally equivalent to $\tilde H_{T/2}$, see Figure 2. These
parts are denoted by $A_{i_1 i_2}$, where $i_k$ is the number of the
part in the $e^k$-direction, $k=1, 2$. That is, we have
$i_1=1,\dots, 4$, $i_2=1,2$. Then, fix a configuration $\ccc_1$ of
these rectangles; the subscript 1 indicates the first generation.
For instance, let $\ccc_1=\{A_{12}, A_{41}\}$.

\begin{figure}[b]

\ifx\JPicScale\undefined\def\JPicScale{1}\fi \unitlength
\JPicScale mm
\begin{center}
\begin{picture}(200,40)(0,0)

\linethickness{0.4mm} \put(0,0){\line(1,0){128}}
\linethickness{0.4mm} \put(0,32){\line(1,0){128}}

\linethickness{0.4mm} \put(0,16){\line(1,0){128}}

\linethickness{0.4mm} \put(0,0){\line(0,1){32}}
\linethickness{0.1mm} \put(8,0){\line(0,1){32}}
\linethickness{0.1mm} \put(16,0){\line(0,1){32}}
\linethickness{0.1mm} \put(24,0){\line(0,1){32}}
\linethickness{0.4mm} \put(32,0){\line(0,1){32}}

\linethickness{0.1mm} \put(40,0){\line(0,1){32}}
\linethickness{0.1mm} \put(48,0){\line(0,1){32}}
\linethickness{0.1mm} \put(56,0){\line(0,1){32}}
\linethickness{0.4mm} \put(64,0){\line(0,1){32}}

\linethickness{0.1mm} \put(72,0){\line(0,1){32}}
\linethickness{0.1mm} \put(80,0){\line(0,1){32}}
\linethickness{0.1mm} \put(88,0){\line(0,1){32}}
\linethickness{0.4mm} \put(96,0){\line(0,1){32}}
\linethickness{0.1mm} \put(104,0){\line(0,1){32}}

\linethickness{0.1mm} \put(112,0){\line(0,1){32}}
\linethickness{0.1mm} \put(120,0){\line(0,1){32}}
\linethickness{0.4mm} \put(128,0){\line(0,1){32}}

\linethickness{0.1mm} \put(0,8){\line(1,0){128}}
\linethickness{0.1mm} \put(0,24){\line(1,0){128}}

\linethickness{0.8mm} \put(96,8){\line(0,1){8}}
\linethickness{0.8mm} \put(120,0){\line(0,1){8}}

\linethickness{0.8mm} \put(24,16){\line(0,1){8}}
\linethickness{0.8mm} \put(0,24){\line(0,1){8}}

\end{picture}
\end{center}
\caption{$A_T$ for  $N=2$.}\label{ff2}
\end{figure}
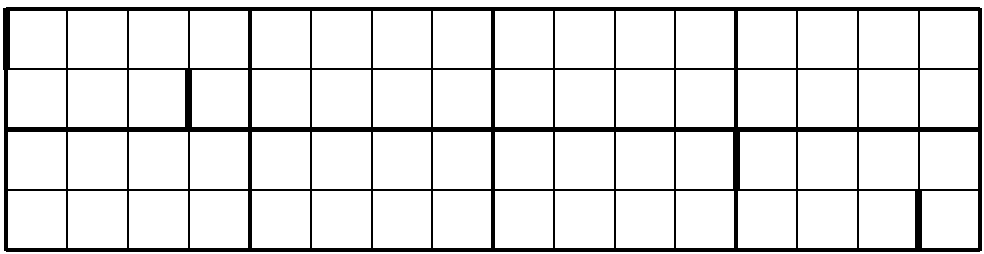

We will properly scale and translate the configuration to get a ``fractal dyadic-type'' set.
 Dyadically partition every $A_{i_1 i_2}$  as we did it for $\tilde H_T$; every part of
  $A_{i_1 i_2 }$ is translationally invariant to $\tilde H_{T/2^2}$.
  These parts are denoted by $A_{i_1 i_2 ;\,  i'_1 i'_2 }$.
  Place the shrunk and translated images of the rectangle from $\ccc_1$ inside each part of $\ccc_1$. This is the configuration $\ccc_2$, i.e.
$\ccc_2=\{A_{12;\, 12}, A_{12;\, 41}, A_{41;\, 12}, A_{41;\, 41}\}$.

Continue in the same manner up to $\ccc_N$. All rectangles from the
\mbox{$N$-th} generation (and hence those appearing in $\ccc_N$) are
the unit squares. Now we set
\begin{center}
$\dsp A_T=\bigcup_{P\in\ccc_N}$ \{ the left hand side of $P$ (in the
$e^1$-direction) \}
\end{center}
and $E_T=A_T\times [0,T]$. Observe that $|\ppr' E_T|=T^2$ and $E_T$
satisfies the assumption $\ti D(\ppr' E)\sim \diam\ppr' E$ from the
formulation of the theorem.

Next, let us estimates $C(E_T)$.  First, let $\mu=m_2|_{E_T}$ where
$m_2$ is the (planar) Lebesgue measure. Obviously, $\mu(E_T)=T^2$.
To estimate $U^-_{\mu}$ on $E_T$ notice that if $z\in A_{12}\times
[0,T]$, then
\[
U^-_{\mu}(z)\geq U^-_{\mu_{12}}(z)+\alpha
\]
where $\mu_{12}$ is $\mu$ restricted to $A_{12}\times [0,T]$ and
$\alpha$ is some universal positive constant. Repeating the same
estimates over all generations and using inductive argument, we see
that the set of $2^N$ obstacles $\{E\}$ that form $E_T$ (and each
obstacle is a rectangle of size $1\times T$ parallel to $OX_2X_3$
coordinate plane) can be partitioned into disjoint groups $\Omega_j,
j=0,\dots N$, such that
$$
U^-_{\mu}(z)\geq \alpha j
$$
for $z\in \Omega_j$ and the cardinality of $\Omega_j$ is $\binom
Nj$. Notice that
\[
\sum_{j: |j-N/2|\lesssim \sqrt{N}} \binom Nj> \delta 2^N, \quad
\delta>0
\]
by Stirling's formula.

Next, we are going to assign $E_T$ a slightly different measure
$\mu^*$ so that $\mu\leq \mu^*$. We take $\mu^*=\mu$ for obstacles
in the groups $\Omega_j$ with $j\ge L=\epsilon N, \epsilon<1/2$ and
\[
\mu^*=N\mu
\]
on $E^j$ for $j< L$. Notice that we now have
\[
U^-_{\mu^*}(z)\gtrsim N=\log T
\]
for all $z\in E_T$ by construction (on $\Omega_0$ one can obviously
improve the bound to $U^-_{\mu}(z)>C>0$). On the other hand,
\[
\sum_{j< L} \binom Nj \lesssim N \binom NL
\]
and, by Stirling's formula once again,
\[
\binom NL \lesssim \frac{N^{-1/2}}{x^N}, \quad x=\epsilon^\epsilon
(1-\epsilon)^{1-\epsilon}>1/2
\]
since $\epsilon<1/2$. Therefore, $\mu^*(E_T)\sim T^2$. Furthermore,
$U^-_{p^+}(z)\lesssim U^-_{\mu^*}(z)/{\log T}$ for any $z\in E_T$,
and, by the maximum principle, for all $z\in\br^3$. Then, taking
$z=(-r,0,0), r\to+\infty$ and recalling \eqref{e91}, \eqref{e92}, we
come to
\[
C^-(E_T)\lesssim \frac{T^2}{\log T}
\]
which finishes the proof.
\end{proof}

\nt{\bf Remark.} When we estimate the modified capacity of the set
$E_T$, we are essentially in the macroscopic regime where the
parabolic approximation takes place. It is instructive to compare
this situation to the one discussed in Taylor-Watson \cite{tw}. For
instance, to compute the parabolic (or heat) capacity of $E_T$, one
can use scaling by $T^2$ in time variable $x_1$ and by $T$ in space
variable $x'\in \br^2$. Of course, we can not apply this trick in
our case but the situation we considered in Theorem \ref{proj}
resembles the calculation in \cite{tw}, Theorem 5 and Example 3. It
is proved there that there are parabolically polar sets in
$[0,1]\times [0,1]$ with full projection on $[0,1]$ (in $e^1$-axis).
The factor $\log T$ appears when computing the Riesz--$1/2$ capacity
of the $2^N$--th approximation to the Cantor--$1/2$ set (see
\cite{tw}, Theorem 5, once again).\medskip

The theorem below is concerned with the case $d=2$. In this
situation, we always have $\ti D(\ppr_y E)=|\ppr_y E|$  (see
\eqref{e93} for the definition of $\ti D$). Therefore,

\begin{theorem}\label{proj_}
Let $E\subset [0,T^2]\times [0,T], \, T>1$, and $|\ppr_y E|>2$. Then
$$
\frac{|\ppr_y E|}{\log |\ppr_y E|} \lesssim C(E)
$$
and this inequality is sharp.
\end{theorem}
Figures \ref{ff1} and \ref{ff2} illustrate the right choice of
contraction $\aleph$ and the construction of the set $E_T$.

\subsection{How does the capacity of a set depend on the speed of the drift?}\label{ss35}
The It\^o diffusion $G_t$ introduced in Theorem \ref{th1}, (2) (see
also \eqref{fk}), was designed to handle the asymptotics of $G(x,0;
i)$ as $|x|\to\infty$. It is clear that the same constructions go
through for any $ik\in i\br_+$ with the only difference that the
kernels $K^\pm$ defining the potential theory are
\[
K^\mp_k (x,y)=\frac 1{2\pi} \frac{e^{-k(|x-y|\pm(x_1-y_1))}}{|x-y|}
\]
where $x,y\in \br^3$. So let $k\in\br^+$ be fixed. Recall the
definitions and notations from Sections \ref{ss31}, \ref{ss32}. Take
a  compact set $E$. Let $\mu_k$ be a maximizer for $C_k$, i.e. the
capacity constructed with respect to $K^-_k$.  That is,
\[
\sup_x \int K^-_k(x,y)d\mu_k (y)=1, \qquad  \mu_k(E)=C_k(E)
\]
The monotonicity of the kernel $K^-_k$ in $k$ implies that $C_k(E)$
grows in $k$.

\section{Some applications}\label{s5}
We can now state some results which are direct corollaries of what
we have proved so far. We only treat $d=2$, the general case can be
handled similarly. Consider
\[
H=H_V=-\Delta+V,
\]
where $V=0$ on $\Omega$. We also assume that $V\geq 0$, $V$ is
measurable and locally bounded. These conditions are sufficient to
define  $H$ as a self-adjoint operator.

The first application is a relatively simple geometric test that
guarantees the presence of the a.c. spectrum. Denote $E=\Omega^c$. We
introduce the following sets via polar coordinates/complex notation:
\[
Q_{n,j}=\left\{ z=re^{i\theta}, 4^n\leq r\leq 4^{n+1},
\frac{2\pi\, j}{ 2^{n}}\leq \theta\leq \frac{2\pi\, (j+1)}{ 2^{n}}
\right\}
\]
where $n=0,1,2,\ldots,$ and $j=0,1,\ldots, 2^n-1$. These curvilinear
rectangles have (polar) ``lengths"  and ``heights" comparable to
$4^n$ and  $2^n$, respectively.  So, they roughly resemble  the
characteristic rectangles discussed in Sections~\ref{ss32},
\ref{s34}. Notice also that the angular projections of $Q_{n+1,j}$
for various $j$ is the diadic collection of arcs which  is the
refinement of arcs from the previous $n$-th generation. Define
\[
E_{n,j}=E\cap Q_{n,j}
\]
For each $n$ and $j$, introduce the Hausdorff content in the
direction $\theta_{n,j}=(2\pi\, j)2^{-n}$ (see Section \ref{ss33}) and
compute $M_{\theta_{n,j},h}(E_{n,j})$. Notice that good upper bounds
for these quantities are sufficient to handle the convergence of
series \eqref{e1012} appearing below.

Let $d_{n,j}(\theta)$ be the arclength distance from $e^{i\theta}$ to the arc
\[
I_{n,j}= [2\pi j\, 2^{-n}, 2\pi(j+1)\, 2^{-n})
\]
 as a set on the unit circle $\mathbb{T}$. Now, we can introduce
\begin{equation}\label{e1012}
\phi(\theta)=\sum_{n=0}^{\infty}
 \sum_{j=0}^{2^n-1} M_{\theta_{n,j},h}(E_{n,j})(2^{-n}+\sin |\theta-\theta_{n,j}|)
\exp\Bigl(- 4^n (1-\cos d_{n,j}(\theta))\Bigr)
\end{equation}

\begin{theorem}
If $\phi(\theta)$ is finite on a set of positive measure, then
$\sigma_{ac}(H)=\mathbb{R}_+$.\label{eas}
\end{theorem}

\begin{proof}
By the standard trace-class perturbation argument,
$\sigma_{ ac}(H)=\sigma_{ ac}(-\Delta+V\cdot \chi_{|x|>R})$
for any $R>0$. Thus, we can always assume that
$$
V(x)=0, \quad |x|<4^{n_0}
$$
for any given $n_0>0$. Let $\phi_{n_0}(\theta)$  be the remainder of  series \eqref{e1012}
$$
\phi_{n_0}(\theta)=\sum_{n=n_0}^{\infty}
\sum_{j=0}^{2^n-1} M_{\theta_{n,j},h}(E_{n,j})(2^{-n}+\sin |\theta-\theta_{n,j}|)
\exp\Bigl(- 4^n (1-\cos d_{n,j}(\theta))\Bigr)
$$
By Egorov's theorem on convergence, one can find a measurable set $\Omega_{n_0}\subseteq \mathbb{T}$ such that
$|\Omega_{n_0}|>\delta>0$ uniformly in $n_0$ and $\phi_{n_0}(\theta)<\epsilon_{n_0}$ on $\Omega_{n_0}$, where $\epsilon_{n_0}\to 0$ as $n_0\to\infty$.

It is sufficient to show that, for large $n_0$, we have
$$\inf_{\theta\in\Omega_{n_0}} a_N(\theta)>\frac 12
$$
uniformly in $N$, where
$a_N$ is amplitude \eqref{as} corresponding to the truncated potential $V_N(x)=V\cdot \chi_{4^{n_0}<|x|<4^{N+1}}$ and $x^0=0$.

Fix $\theta\in \Omega_{n_0}$. By (\ref{lapa}), we need to prove
\[
\mathbb{P}\Bigl(G_t^\theta {\rm \ does\ not \ hit\ } E\cap \{4^{n_0}\leq |z|\leq 4^{N+1}\}\Bigr)>\frac 12
\]
uniformly in $N$. Instead, we will show that
\[
\mathbb{P}\Bigl(G_t^\theta {\rm \  hits\ } E\cap \{4^{n_0}\leq |z|\leq 4^{N+1}\}\Bigr)
\] is small uniformly in $N$ provided $n_0$ is large enough. Here, $G^\theta_t=te^{i\theta}+B_t$. Obviously,
\[
\mathbb{P}\Bigl(G_t^\theta {\rm \  hits\ } E\cap \{4^{n_0}\leq |z|\leq 4^{N+1}\}\Bigr)\leq \sum_{n=n_0}^N \sum_{j=0}^{2^n-1}
 \mathbb{P}\Bigl(G_t^\theta {\rm \  hits\ } E_{n,j}\Bigr)
\]
Fix $n$ and consider $E_{n,j}$. Let $d\nu_{n,j}$ be the maximizer in the definition of $C_\theta(E_{n,j})$.
Then,
\begin{eqnarray*}
\mathbb{P}\Bigl(G_t^\theta {\rm \  hits \ } E_{n,j}\Bigr)=\int
K^-_\theta(0,\xi)d\nu_{n,j}(\xi) &\lesssim& 2^{-n}\int
e^{-|\xi|(1-\cos  \alpha(\xi,\theta))}d\nu_{n,j}(\xi)
\end{eqnarray*}
where $\alpha(\xi,\theta)$ is the angle between $\xi$ and $\theta$.
We always have $C_\theta(E_{n,j})\lesssim 4^n$. If $j$ is such that,
say, $\alpha(\xi,\theta)>\frac 1{10}$, then
\begin{equation}
2^{-n}\int e^{-|\xi|(1-\cos  \alpha(\xi,\theta))}d\nu_{n,j}(\xi) \lesssim 2^{-n}e^{-4^{n-1}}4^n \label{oe1}
\end{equation}
Hence, we are only interested in small angles $\alpha(\xi, \theta)$ for which
\begin{equation}
2^{-n}\int e^{-|\xi|(1-\cos  \alpha(\xi,\theta))}d\nu_{n,j}(\xi) \lesssim 2^{-n}\exp \Bigl(-4^n(1-\cos d_{n,j}(\theta))\Bigr) C_\theta(E_{n,j})\label{oe2}
\end{equation}
Proposition \ref{content} yields
\[
C_\theta(E_{n,j})\lesssim M_{\theta,h}(E_{n,j})
\]
Now, obvious geometric considerations based on rotation of the coverings give
\[
M_{\theta,h}(E_{n,j})\lesssim M_{\theta_{n,j},h} (1+2^n\sin |\theta_{n,j}-\theta|)
\]
So, summation gives
\[
\mathbb{P}\Bigl(G_t^\theta {\rm \,  hits\, } E\cap \{4^{n_0}\leq |z|\leq 4^{N+1} \}\Bigr)\lesssim \phi_{n_0}(\theta)+\sum_{n=n_0}^\infty 4^n e^{-4^{n-1}} \to 0
\]
for $n_0\to\infty$ uniformly in $N$.
\end{proof}

\nt
{\bf Remark.} The following estimate on $\phi(\theta)$ is straightforward
\[
\phi(\theta)\lesssim \sum_{n=n_0}^{\infty}2^{-n}
 \sum_{j} M_{\theta_{n,j},h}(E_{n,j})(1+k(n,j,\theta))
\exp\Bigl( -\alpha k^2(n,j,\theta)\Bigr) + \overline{o}(1)
\]
where $0\le\alpha<\frac 12$  and $k(n,j,\theta)$ is the number of
the full diadic intervals on $\mathbb{T}$ between $\theta$ and
$I_{n,j}$.  An advantage of  Theorem \ref{eas} is its relative
simplicity since any reasonable covering of $E$ provides an upper
bound for the Hausdorff content.

\medskip
The second application deals with the case when the Hausdorff
content is  too rough an instrument to estimate the capacity. We are
going to construct a set $E$ such that any ray issued from the origin
intersects it infinitely many times and yet the associated
Schr\"odinger operator has the a.c. spectrum that is equal to $\br_+$.

\medskip\nt
{\bf Example.} Consider some large $T=2^n$ and the annulus
$\{T^2\leq |z|\leq 2T^2\}$. Take intervals $I_{j}=[(2\pi j)\,
T^{-1}, 2\pi (j+1)\, T^{-1})$, $j=0,1,\ldots, T-1$,  and sets
$Q_j=\{z=re^{i\theta}: T^2\leq r\leq 2T^2, \theta\in I_j\}$. In each
$Q_j$,  place a set $E_j$ in a way similar to that one of example
given in the proof of Theorem~\ref{proj}. The only difference now is
that we dyadically divide the angle and the radius. Also, instead of
vertical segments of unit length we take the corresponding sectorial
rectangles of the size comparable to one. In this construction, all
$E_j$ can be obtained from, say, $E_0$ by rotations by angles $(2\pi
j)\, T^{-1}, \ j=1,\ldots T-1$. Finally, set $E_T=\bigcup_j E_j$.

\begin{lemma}
As $T\to\infty$, we have
\[
\delta_T=\sup_{\theta\in [0,2\pi)} \mathbb{P}\left(G^\theta_t=(te^{i\theta}+B_t)
{\rm \ hits \ } E\right)\to 0
\]
\end{lemma}

\smallskip\nt {\it Sketch of the proof.}\
Take an arbitrary $\theta\in\mathbb{T}$. For each interval $I_j$, let $k(j,\theta)$ be the total number of intervals $\{I_m\}$ between $\theta$ and $I_j$.
Estimating just like in (\ref{oe1}) and (\ref{oe2}), we have
\begin{equation}
\mathbb{P}\left( G^\theta_t{\rm \ hits \ }  \bigcup_{j:
k(j,\theta)>k_0} E_j\right)\lesssim e^{-CT^2}+\sum_{k(j,\theta) \geq k_0}
e^{-Ck^2(j,\theta)}\label{part1}
\end{equation}
On the other hand, for any fixed value $k(j,\theta)$,  the bounds
from the proof of  Theorem \ref{proj} apply as long as $T$ is large
enough and one has
\[
\lim_{T\to\infty} C_\theta(E_j)/T\to 0
\]
That, in turn, implies
\[
\mathbb{P}(G_t^\theta \, {\rm hits}\, E_j)\to 0, \,\, {\rm as}\,\,
T\to\infty
\]

 Then (\ref{part1}) finishes the proof. \hfill
$\Box$

\medskip
Now, it is enough to take $T_k, \ \lim_{k\to \infty} T_k=\infty$ with the property
\[
\sum_k \delta_{T_k}<1/2
\]
and consider $V\geq 0$ supported on $E=\bigcup\limits_k E_{T_k}$.
Then we have
\[
\sup_{\theta\in [0,2\pi)} \mathbb{P}\Bigl( te^{i\theta}+B_t  \ {\rm
hits} \  E \Bigr)< 1/2
\]
and thus $\sigma_{ac}(H)=\mathbb{R}_+$. Notice that by construction any ray issued from the origin intersects the support of $V$ infinitely many times.
\section*{Appendix A: the upper bounds on the harmonic measure}\label{s4}

In this Appendix, we  adjust the well-known Carleman estimates on
harmonic measure to our case. These bounds can be used to prove that
the amplitude $a_{x^0}=0$ and so they do not have immediate
consequences for the study of spectral types. Nevertheless, we prove
them to emphasize the ``parabolic nature" of the domains for which
the transition from $a_{x^0}=0$ to positive values occurs.

Consider the following problem. Let
$\Omega\subset\br^3$ be a bounded connected domain (but not
necessarily simply connected) with smooth boundary.
Assume also that $B(0, 1)\subset \Omega$. Then, define
$\Omega_R=\Omega\cap \{x: -1<x_1<R\}$ for $R$ large, and let
$\Gamma_R=\partial \Omega_R$. We consider the modified harmonic
measure $\omega(. ,A,\Omega_R)$ \eqref{e95} but with respect to the
bounded $\Omega_R$ and $A\subset \Gamma_R\cap \{x: x_1=R\}$. Recall
the probabilistic interpretation of $\omega(0,A,\Omega_R)$: it is
the probability of a random trajectory $G_t=0+t\cdot e^1+B_t$ to hit
the boundary $\Gamma_R$ for the first time at $A$. Above,  $B_t$ is
the three-dimensional Brownian motion. We are interested in the case
when $\Omega$ is elongated in the $e^1$-direction. To make the
explanation more intuitive, we denote the first component of the
vector $x=(x_1, x_2, x_3)\in\br^3$ by $t,\ t=x_1$. As before $x=(t,
x'), x'\in \br^2$.

We start with some notations. Let $\Theta$ be a ($x'$-planar)
bounded open set with piece-wise smooth boundary and
$$
\cw(\Theta)=\{f\in W^{1,2}(\Theta): f|_{\partial\Theta}=0, \ f\not=
0\}
$$
Consider the operator $-\Delta_D$ where $\Delta_D$ is the Laplacian
on $\Theta$ with
 Dirichlet boundary condition on $\pt\Theta$. It is well known
 that its spectrum  is discrete,
$\sigma(-\Delta_D)=\{\lambda_1<\lambda_2\leq \lambda_3\leq
\ldots\}$. The value $\lambda_1>0$ is called principal eigenvalue
and we denote it by $\lambda=\lambda_\Theta$. The classical
Courant-Hilbert principle \cite[Sect. 4.2, 6.7.9]{ch}
 says that
\begin{equation}
\min_{f\in\cw(\Theta)}\frac{\int_\Theta |\nn f|^2 dy'}{\int_\Theta
f^2 dy'}=\lambda_\Theta
\end{equation}

Denote the section of $\Omega_R$ by the plane $\{x: t=s\}$ by
$\Theta_s$. Clearly, $\Theta_s$ is a union of a finite number of
connected components $\Theta^j_s, \ j=1,\dots,N$.  Let
$$
\lambda(s)=\min_{j} \lambda_{\Theta^j_s}
$$
Then, for $f\in \cw(\Theta_s)$,
\begin{eqnarray}\label{e99}
\int_{\Theta_s}|\nn' f|^2 dy'&\geq&\sum_j \int_{\Theta^j_s}|\nn' f|^2 dy'\geq \lambda \sum_j \int_{\Theta^j_s}f^2 dy'\\
&=&\lambda \int_{\Theta_s}f^2 dy' \nonumber
\end{eqnarray}
where $\nn'$ is the gradient with respect to $x_2,x_3$.

\begin{theorem}\label{th3}  We have
\begin{equation}
\omega(0,A,\Omega_R)\leq C_1 |A|^{1/2}e^R\left(1+\int_0^R
\exp\lt(2\int_0^t \sqrt{\lambda(u)+1}\, du\rt)
dt\right)^{-1/2}\label{ohho1}
\end{equation}
In particular,
\begin{equation}\label{ohho2}
\omega(0,A, \Omega_R)\leq C_2
|A|^{1/2}\exp\left(R-\int_0^{R-1}\sqrt{\lambda(u)+1}\, du\right)
\end{equation}
\end{theorem}

\begin{proof}
The proof  essentially follows Carleman \cite{carl}, Haliste
\cite{ha}; see also \cite[Appendix G]{GM}. Recall that
$\Omega_s=\Omega_R \cap \{x: t<s\}$. The function $u(z)=\omega(z,
A,\Omega_R)$ satisfies
\[
\Delta u+2u_{x_1}=0, \quad u|_{\Gamma}=\chi_A
\]
Consider
\begin{equation}\label{e991}
\phi(t)=\int_{\Theta_t} u^2(t,y')dy', \quad A(t)=\int_{\Omega_t}
|\nabla u(s,y')|^2 ds dy'
\end{equation}
Applying Gauss-Ostrogradsky formula to $u\cdot \nn u$ on $\Omega_t$,
we have the following relations
\[
A(t)=\int_{\Theta_t} u(t,y')u_t(t,y')dy'+\int_{\Theta_t}
u^2(t,y')dy'=\phi'(t)/2+\phi(t)
\]
and
$$
A'(t)=\int_{\Theta_t} |\nn u(t,y')|^2 dy'= \int_{\Theta_t}
(u^2_{x_1}(t,y') + u^{2}_{x_2}(t,y')+u^2_{x_3}(t,y')) dy'
$$
Recalling the definition of $\lambda$ and \eqref{e99}, we see
\begin{eqnarray*}
\int_{\Theta_t}(u^2_{x_2}(t,y')+u^2_{x_3}(t,y')) dy' &=& \int_{\Theta_t} |\nn' u|^2 dy'\\
&\geq& \lambda(t)\int_{\Theta_t} u^2 dy' =\lambda(t)\phi(t)
\end{eqnarray*}
By Cauchy-Schwarz,
\[
\phi'(t)\leq 2\sqrt{\phi(t)}
\left(\int_{\Theta_t}u^2_{x_1}(t,y')dy'\right)^{1/2}
\]
Thus, we have the following differential inequality
\[
\phi''+2\phi'\geq 2\lambda(t)\phi+\frac{{\phi'}^2}{2\phi}
\]

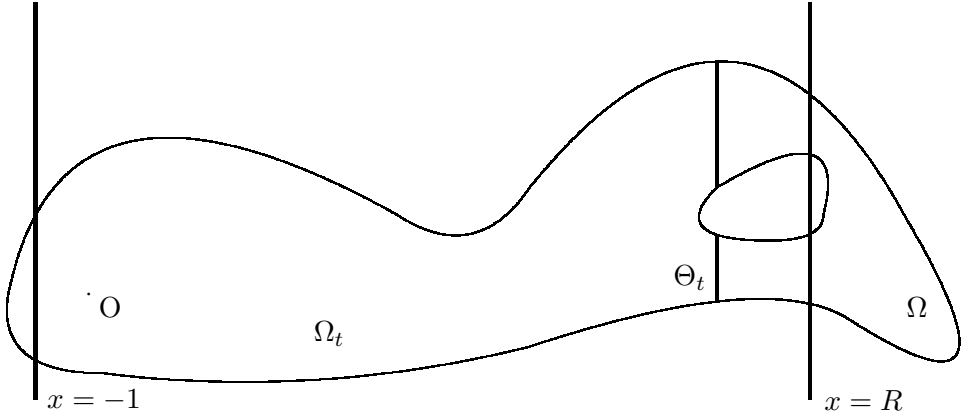
\begin{figure}

\begin{center}
\begin{picture}(200,200)(70,0)
\renewcommand{\qbeziermax}{1000}

\qbezier(5,40)(0,10)(40,10) \qbezier(5,40)(25,140)(150,70)

\qbezier(150,70)(180,50)(200,80) \qbezier(200,80)(280,180)(340,70)

\qbezier(200,20)(290,50)(320,30)

\qbezier(320,30)(390,-15)(340,70)

\qbezier(40,10)(120,0)(200,20)
\put(125,25){\makebox(0,0)[cc]{$\Omega_t$}}
\put(345,35){\makebox(0,0)[cc]{$\Omega$}}

\qbezier(270,80)(320,110)(310,70)

\qbezier(270,80)(250,60)(290,60)

\qbezier(290,60)(310,60)(310,70)

\linethickness{0.5mm} \put(305,0){\line(0,1){150}}
\linethickness{0.5mm} \put(15,0){\line(0,1){150}}
\put(325,0){\makebox(0,0)[cc]{$x=R$}}
\put(37,0){\makebox(0,0)[cc]{$x=-1$}}
\put(35,40){\circle*{1}} \put(43,35){\makebox(0,0)[cc]{O}}

\linethickness{0.3mm} \put(270,37){\line(0,1){25}}
\linethickness{0.3mm} \put(270,80){\line(0,1){47}}

\put(260,46){\makebox(0,0)[cc]{$\Theta_t$}}

\end{picture}
\end{center}
\caption{ Carleman's method ($d=2$).}\label{f3}
\end{figure}

\nt Make the substitution $\phi=\vp e^{-2t}$. Then,
$$
\vp''\geq 2(\lambda +1)\vp+ \frac{\vp'^2}{2\vp}
$$
This inequality implies $\varphi''>0$ for $t\in (-1,R)$ and,  since
$\varphi(-1)=0$ and $\varphi(R)>0$, we must have $\varphi'(t)>0$.
After multiplying by $2/\vp$ and complementing to a full square, we
obtain
$$
\lt(\frac{\vp''}{\vp'}\rt)^2 -
\lt(\frac{\vp''}{\vp'}-\frac{\vp'}{\vp}\rt)^2\geq 4(\lambda +1)
$$
Since $\varphi''/\varphi'>0$, we have $\vp''\geq 2\sqrt{\lambda
+1}\, \vp'$. Set $\mu=2\sqrt{\lambda+1}$ and
$$
\psi(t)=\int^t_{-1}\exp\lt(\int^u_{-1} \mu(s)ds \rt) du
$$
for the sake of brevity. One has $\psi''=\mu \psi'$. Hence,
$$
\left(\log
\frac{\vp'}{\psi'}\right)'=\frac{\vp''}{\vp'}-\frac{\psi''}{\psi'}\geq
0
$$
so the function $\vp'/\psi'$ is non-decreasing and, for $-1<x<t$
$$
\vp'(x)\psi'(t)\leq \vp'(t)\psi'(x)
$$
Integrating in  $x$ and $t$ from $-1$, we see that
$\vp(x)\psi(t)\leq \vp(t)\psi(x)$. Putting this a bit differently,
$$
\phi(x)\leq \frac{\psi(x)}{\psi(t)}\phi(t) e^{2(t-x)}
$$
Since $\phi(R)=|A|$, we have
$$
\phi(0)\lesssim |A| e^{2R} \left(1+\int_0^{R}\exp \lt(2\int^t_0
\sqrt{\lambda +1} ds\rt) dt \right)^{-1}
$$
Recalling \eqref{e991} and Harnack's principle,
$$
\omega(0,A,\Omega_R)\lesssim |A|^{1/2} e^R  \left(1+\int_0^{R}\exp
\lt(2\int^t_0 \sqrt{\lambda +1} ds\rt) dt \right)^{-1/2}
$$
As for the second inequality claimed in the theorem, we only need to
use
\begin{eqnarray*}
\int_0^R \exp \lt(2\int^t_0  \sqrt{\lambda +1} ds  \rt)dt\ge  \exp
\lt (2\int^{R-1}_0 \sqrt{\lambda +1} ds\rt)
\end{eqnarray*}
\end{proof}

\nt{\bf Remark.} The case of a wide long strip shows that the first
estimate is essentially sharp.

\medskip
\nt{\bf Remark.} For $d=2$, the bound on the harmonic measure is the
same but the principal eigenvalue equals to
$$
\l(t)=\lt(\frac\pi{l(t)}\rt)^2
$$
where $l(t)$ is the length of the longest interval among  intervals
forming  $\Theta_t$. If $d=3$ and $\Theta_t$ is simply-connected,
the principal eigenvalue is controlled by the inradius:
$\lambda(t)\sim r^{-2}(t)$. In general, if $d>3$ or $d=3$ and
$\Theta_t$ is not simply-connected, one can use $\lambda(t)\geq
h^2(t)/4$ where $h(t)$ is the Cheeger's constant of $\Theta_t$
\cite{grie}. Unfortunately, it is neither easy to compute nor
to estimate in practice.

\section*{Appendix B: the lower bounds on harmonic measure}

In this section, we will estimate the harmonic measure from below.
Recall that  a set $A\subset \br^d$ is Steiner symmetric with respect to a $(d-1)$-dimensional plane $p$, if $s\cap A$ is an interval symmetric w.r. to $p$ for any straight line $s$ perpendicular to $p$ provided $s\cap A\not =\emptyset$, see P\'olya-Szeg\H o \cite{ps}. We say that $A$ is Steiner symmetric (without mentioning $p$) if $A$ is Steiner symmetric w.r. to coordinate planes.

Let $\Theta$ be a Steiner symmetric domain with piece-wise smooth boundary lying in the $x'$-plane, and
 $\ti \Theta_t=\ti k(t) \cdot \Theta$, where $\ti k :\br\to \br_+ $ is a monotonically decreasing smooth function, $\ti k(0)=1$. Denote by $\lambda=\lambda(0)$ the principal eigenvalue of the operator $-\Delta_D$ on $\Theta$. Let
$\psi$ be the corresponding eigenfunction normalized by $\psi(0,0)=1$.

\begin{lemma}
Let $\Omega=\{(t, x'): 0<t<\infty, \ x'\in \ti \Theta_t\}$ and $u$  solve the following equation
\[
u_t=\frac12 \Delta u
\]
on $\Omega$. Above, the boundary conditions are \ \ $u(0, x')=1, x'\in\Theta$, and $u|_{\partial\ti\Theta_t}=0, t>0$.
Then
\[
u(t,0,0)\geq \exp\left(-\frac{\lambda}{2}\int_0^t \frac{d\tau}{\ti k^2(\tau)}\right)
\]
\end{lemma}
\begin{proof}
Take the trial function
\begin{equation}\label{snizu1}
\tilde{u}(t,x')=\psi\lt(\frac{x'}{\ti k(t)}\rt)\, \exp\left(-\frac \lambda{2}
\int_0^t \frac{d\tau}{\ti k^2(\tau)}\right)
\end{equation}
Since $\psi=0$ on $\pt\Theta$, the function $\ti u(t, .)$ vanishes on $\partial \ti\Theta_t$ for any $t$. Furthermore, we have by \cite{ha}, Lemma 7.3, that  $x'\cdot\nabla \psi\leq 0$ and, in particular, $\psi$ has its maximum at the origin. Consequently,  $\tilde u(t, .)$ also attains its maximum at the origin (of the $x'$-plane); the maximum is easy to compute using \eqref{snizu1}. Moreover,
\[
\frac 12 \Delta\tilde u -\tilde u_t=\frac{\ti k'}{\ti k^2}\,
(x'\cdot \nabla \psi)\, \exp\left(-\frac{\lambda}{2} \int_0^t
\frac{d\tau}{\ti k^2(\tau)}\right) \geq 0
\]
By the maximum principle for (parabolic) subharmonic functions
(\cite{landis}, Theorem~2.1), we get $\ti u\leq u$,  and the
bound is proved.
\end{proof}

Let us now change the set up a little: for $\Theta$ as above, define
$\Theta_t=k(t) \cdot \Theta$, where $k: \br\to\br_+$ is monotonically increasing smooth function, $k(0)=1$.

\begin{theorem} We have
\begin{equation}
\omega(0,\Theta_R, \Omega_R)\geq C(\delta)
\exp\left(-\frac{\lambda}{2}\int_2^R
\frac 1{k^2(t-t^\delta)} dt\right)\label{snizu}
\end{equation}
where $\lambda=\lambda(0)$ is the principal eigenvalue of  $-\Delta_D$ on $\Theta$ and $\frac12 <\delta<1$ is a fixed parameter.
\end{theorem}
\begin{proof}
Recall the probabilistic interpretation of $\omega(0,\Theta_R,
\Omega_R)$. We write $G_t=t\cdot e^1+B(t)=(t+B_1(t), B'(t))$ and $B_j$ are
independent one-dimensional Brownian motions (warning: $B'(t)=(B_2(t), B_3(t))$, and not a derivative).  Then,
$$
\omega(0,\Theta_R, \Omega_R)\gtrsim
\mathbb{P}_{B'}\Bigl(B'(t)\in \Theta_{k(t-t^\delta)},
\forall t>2\Big| E\Bigr)\, \mathbb{P}_{B_1}(E)
 $$
where $E$ is an event $\{|B_1(t)|<t^\delta$ for all $t>2\}$. By the
law of iterated logarithm, $\mathbb{P}_{B_1}(E)>0$ and thus the problem is
reduced to estimating the solution for the parabolic equation from the above lemma.
Apply it to the function $\tilde k(t)=k(t-t^\delta)$ and make the change of variable $t\mapsto R-t, t\in (0,R)$; the theorem follows.
\end{proof}

\nt
{\bf Remark.} Comparison of \eqref{snizu} to (\ref{ohho2}) shows that the estimates
are sharp in a certain sense for the Steiner symmetric
domains monotonically opening at infinity. The case when the scaling function $k$
is decreasing can be handled by adjusting the methods of \cite{ha},
Theorem 7.1.

\medskip
If the function $k$ is not monotonic, proving the lower bounds
for the harmonic measure is an interesting problem.

The estimate that we just obtained is not so sharp if one tries to
prove that the harmonic measure is positive uniformly in $R$. For
example, it does not allow one to recover the law of iterated logarithm.
Instead, one can use estimates from Theorem \ref{eas}. As an
alternative to this method, we suggest the following approach.

We want to obtain the bound from below on $\omega(0,A,\Omega_R)$
which is uniform in $R$. Let us first consider domains embedded in
$\br^2$, i.e. $\Omega=\{(t, x) : t\in (-1,\infty), |x|<\theta(t)\}$
where $x$ is scalar and   $\theta$ is a positive smooth function.
The boundary of $\Omega$ is denoted by $\Gamma$. Let $G_t=(t+B_1(t),
B_2(t))$, where $B_j$ are independent one-dimensional Brownian
motions. We now want to find the conditions on $\theta$ which
guarantee that $\mathbb{P}(G_t \text{ does not hit } \Gamma\ \forall
t)>0$.

For this purpose, we will use the argument from the proof of the law
of iterated logarithm \cite{Yuval}, Theorem 5.1.  Take any
$\epsilon>0$ and $q>1$ and define
\begin{eqnarray*}
C_q=2\log q, \qquad && l_n=\sqrt{(C_q+\epsilon)q^{n+1}\log_q n},
\end{eqnarray*}
For a given $q$, we introduce the following characteristic
parameters of $\theta$:
\begin{eqnarray*}
 \kappa_n=\min_{[q^{n},q^{n+1}]} \theta(t)  &&
\end{eqnarray*}
where $n=1,2,\ldots$.
\begin{theorem} If
\[
\sum_n \frac{q^{n/2}}{\kappa_n} \exp\Bigl(
-\frac{\kappa_n^2}{2q^{n+1}}\Bigr)<\infty
\]
then  \ $\mathbb{P}(G_t \text{ does not hit } \Gamma\ \forall t)>0$.
\label{below}
\end{theorem}

\begin{proof}
Assume that this is not the case. Then, for any large $T>0$, the
trajectory $G'_t=(T+t+B_1(t),B_2(t))$ hits $\Gamma$ almost surely
(warning: once again, $G'_t$ is not a derivative in time). Take $T$
large and introduce the geometric decomposition in time $
I_n=[q^n,q^{n+1}), n\in \mathbb{N}$. We also need
$$
J_n=[q^n+l_n,q^{n+1}-l_{n}], \quad D_n=[q^n-l_{n-1},q^n+l_n]
$$
Consider the random numbers
$$
\alpha_n=\max\limits_{J_n} |B_1(t)|, \ \beta_n=\max\limits_{J_n}
|B_2(t)|, \  \gamma_n=\max\limits_{D_n} |B_1(t)|, \  \delta_n=\max\limits_{D_n} |B_2(t)|
$$
We have
\begin{eqnarray}
\mathbb{P}(G_t' {\rm  \,\,hits\,\, } \Gamma ) &\leq&
\mathbb{P}\lt(G_t' \,\,{\rm hits}\,\, \Gamma\,\, {\rm for\,\,
some}\,\, t\in [0,q^{n_0}]\rt)   \nonumber \\
&+& \sum_{n>n_0} \mathbb{P}\lt(G_t' \,\,{\rm hits}\,\, \Gamma\,\, {\rm for\,\,
some}\,\, t\in J_n \rt)  \nonumber \\
&+& \sum_{n>n_0} \mathbb{P}\lt(G_t'\,\,{\rm hits}\,\, \Gamma\,\, {\rm for\,\,
some}\,\, t\in D_n \rt) \label{rhs}
\end{eqnarray}
Recall that \cite{Yuval}, Theorem 2.18, for $a>0$,
\[
\mathbb{P}(\max_{t\in [0,T]} B(t)>a)=2\mathbb{P}(B(T)>a),
\]
Then
\begin{eqnarray*}
\mathbb{P}\Bigl(G_t' \,\,{\rm hits}\,\, \Gamma\,\, {\rm for\,\,
some}\,\, t\in J_n\Bigr)&=&\mathbb{P}\Bigl(G_t' \,\,{\rm hits}\,\, \Gamma\,\, {\rm for\,\,
some}\,\, t\in J_n\Big| E_n\Bigr)\mathbb{P}(E_n)\\
&+&
\mathbb{P}\Bigl(G_t' \,\,{\rm hits}\,\, \Gamma\,\, {\rm for\,\,
some}\,\, t\in J_n\Big| E_n^c\Bigr)\mathbb{P}(E_n^c)
\end{eqnarray*}
where $E_n$ is the event that $\{\alpha_n<l_n\}$.
An easy computation shows
\begin{equation}\label{ee01}
\mathbb{P}(E_n^c)\lesssim \exp\Bigl( -(\log q+\epsilon/2)\log_q n\Bigr)\in
\ell^1
\end{equation}
and the estimate on $\beta_n$ gives
\[
\mathbb{P}\Bigl(G_t' \,\,{\rm hits}\,\, \Gamma\,\, {\rm for\,\,
some}\,\, t\in J_n\Big| E_n\Bigr)\lesssim
\frac{q^{n/2}}{\kappa_n}\exp\Bigl(
-\frac{\kappa_n^2}{2q^{n+1}}\Bigr)\in \ell^1
\]
by assumption.

Similarly, the estimates for $\gamma_n$ and $\delta_n$ yield
\begin{eqnarray*}
\mathbb{P}\lt(G_t' \,\,{\rm hits}\,\, \Gamma\,\, {\rm for\,\,
some}\,\, t\in D_n\rt)&\lesssim& \nu_n+\frac{\sqrt{q^{n}+2l_n}}{\kappa_{n-1}}\exp\lt(
-\frac{\kappa_{n-1}^2}{2(q^{n}+2l_n)}\rt)\\
&+&\frac{\sqrt{q^{n}+2l_n}}{\kappa_n}\exp\lt(
-\frac{\kappa_n^2}{2(q^{n}+2l_n)}\rt)
\end{eqnarray*}
where $\nu_n\in \ell^1$ as in \eqref{ee01}. The last two terms are in $\ell^1$ as well by our assumption.
 Thus, one can choose $n_0$ large enough to guarantee that
the second and the third terms in the right hand side of (\ref{rhs})
are arbitrarily small. Then, as long as $n_0$ is fixed, the first
term can be made arbitrarily small by taking $T$ large, and the theorem is proved.
\end{proof}

\nt
{\bf Remark.} If $\theta(t)=\sqrt{\gamma t\log\log t}$, then
the law of iterated logarithm shows our result is sharp in some
sense (choose $\gamma<2$ and $q=1+\epsilon$, $\epsilon>0$ small enough).
\medskip

We now turn to the three-dimensional case and consider
$G(t)=t\cdot e^1 +B(t)=(t+B_1(t), B_2(t), B_3(t))$,  and
$$
\Omega=\big\{(t, x_2, x_3): t\in (-1,\infty),  |x_2|< \theta_2(t), |x_3|< \theta_3(t) \big\},
$$
$\theta_{2, 3}$ being positive smooth functions. For $q_2, q_3>1$,
define
\begin{eqnarray*}
\kappa_{j,n}=\min_{[q_j^{n},q_j^{n+1}]} \theta_j(t)  &&
\end{eqnarray*}
where $j=2,3$, and $n=1,2,\ldots$.

The proof of Theorem \ref{below1} is based  on Theorem \ref{below} and the independence of
 Brownian motions $B_2$ and $B_3$.
\begin{theorem}\label{below1} If
$$
\sum_n \frac{q_j^{n/2}}{\kappa_{j,n}} \exp\Bigl(
-\frac{\kappa_{j,n}^2}{2q_j^{n+1}}\Bigr)<\infty
$$
for $j=2,3$, then  \  $\mathbb{P}(G_t \text{ does not hit } \Gamma\
\forall t)>0$.
\end{theorem}

\section*{Appendix C: the case $d>3$}\label{s7}

Let, as before,  $x=(x_1,x')\in\br^d$ and $x'\in\br^{d-1}$. The Green's function of $(-\Delta+1)^{-1}$ on $L^2(\br^d)$ is given by
\[
G_0(x,y)=C_d|x-y|^{-\nu}K_{\nu}(|x-y|)
\]
and
\[
G_0(x,y)\approx \left\{
\begin{array}{lcl}
\dsp C_d^1\, e^{-|x-y|}{|x-y|^{-(d-1)/2}},& &|x-y|\to\infty,\\\\
 C_d^2\, |x-y|^{-(d-2)},& & |x-y|\to 0,
\end{array}
\right.
\]
see Section \ref{ss21} for the values of the constants $C^{1,2}_d$.
Once again, introduce potentials
\[
K^-(z,\xi)=2G_0(z,\xi) e^{\xi_1-z_1}, \qquad K^+(z,\xi)=K^-(\xi, z)
\]
and compare this to \eqref{e9501}. The corresponding capacity $C=C^\pm$ is then defined as in \eqref{def1} or \eqref{second}.

As for Hausdorff content, proceed like in Section \ref{ss33}. Set
$\Pi_r=[0,r]^d$ for $0<r\leq 1$, and
$\Pi_r=[0,r^2]\times[0,r]^{d-1}$ for $r>1$. We say that $H(r)$ is
the characteristic set if it can be translated into the parallelepiped
$\Pi_r$.  The Hausdorff content $M_h(E)$ of a set $E\subset \br^d$
is defined by relation \eqref{e9502} with
\begin{equation*}
h(r)=\left\{
\begin{array}{cc}
r^{d-2},& 0<r\leq 1\\
r^{d-1},& r>1
\end{array}\right.\label{mes_}
\end{equation*}
\begin{proposition}\label{content_}
For any set $E\subset\br^d$ and the above measure function $h$, we have
\[
C(E)\lesssim M_h(E)
\]
\end{proposition}
Furthermore, the definitions of contractions $\aleph$ and $\ppr'$ do
not change. Obviously, for $E\subset \br^d$,
 the expression $|\ppr' E|$ refers now to $(d-1)$-dimensional volume
 etc. With these conventions, the formulations of
 Theorems \ref{projx}, \ref{proj}, \ref{th3} are the same word-for-word.

\bigskip\nt
{\bf Acknowledgement.} \rm The first author was supported
by Alfred P. Sloan Research Fellowship and by the NSF grant
DMS-0758239. The second author was supported by ANR grants ANR-07-BLAN-024701 and ANR-09-BLAN-005801. We are grateful to Yuval Peres for informing us of \cite{tw}
and to Fedor Nazarov for numerous helpful discussions on the subject of this
paper.

\end{document}